\documentclass[11pt, reqno]{amsart}
\usepackage{graphicx,amssymb,amsmath,amsthm}
\usepackage{enumerate}
\usepackage{dsfont}
\usepackage[colorlinks, citecolor=red]{hyperref}
\usepackage{comment,cite,color}
\usepackage{cite,color}

\usepackage{mathrsfs}
\usepackage{epsfig}
\usepackage{lscape}
\usepackage{subfigure}
\usepackage[outdir=./]{epstopdf}

\usepackage{caption}
\usepackage{bm}
\usepackage{algorithm}
\usepackage{algpseudocode}
\usepackage[english]{babel}

\usepackage{multirow}

\usepackage{fullpage}
\usepackage[margin=1in]{geometry}
\newcommand{\dkh}[1]{\left\{#1\right\}}

\newcommand{\nj}[1]{\langle {#1} \rangle}

\newcommand{\norm}[1]{\|{#1}\|}

\newcommand{\norms}[1]{\|{#1}\|}
\newcommand{\abs}[1]{\lvert#1\rvert}

\newcommand{\R}{{\mathbb R}}
\newcommand{\Rn}{{\mathbb R}^n}

\newcommand{\T}{\top}

\newcommand{\vx}{{\bm x}}

\newcommand{\vz}{{\bm z}}

\newcommand{\va}{{\bm a}}

\newcommand{\RNum}[1]{\uppercase\expandafter{\romannumeral #1\relax}}

\numberwithin{equation}{section}

\newtheorem{definition}{Definition}[section]

\newtheorem{theorem}[definition]{Theorem}

\newtheorem{remark}[definition]{Remark}

\newtheorem{example}[definition]{Example}

\newtheorem{thm}{Theorem}[section]
\newtheorem{lem}{Lemma}[section]
\theoremstyle{remark}
\newtheorem{rem}{Remark}[section]
\newtheorem*{rem*}{Remark}

\date{}

\begin{document}

\author{Jian-Feng Cai}
\thanks{J. F. Cai was supported in part by Hong Kong Research Grant Council grants 16309518, 16309219, 16310620, 16306821.}
\address{Department of Mathematics, The Hong Kong University of Science and Technology, Clear Water Bay, Kowloon, Hong Kong, China}
\email{jfcai@ust.hk}

\author{Meng Huang}
\address{Department of Mathematics, The Hong Kong University of Science and Technology, Clear Water Bay, Kowloon, Hong Kong, China}
\email{menghuang@ust.hk}

\author{Dong Li}
\address{SUSTech International Center for Mathematics and Department of Mathematics, Southern University of Science and Technology, Shenzhen, China}
\email{lid@sustech.edu.cn}

\author{Yang Wang}
\thanks{Y. Wang was supported in part by the Hong Kong Research Grant Council grants 16306415 and 16308518.}
\address{Department of Mathematics, The Hong Kong University of Science and Technology, Clear Water Bay, Kowloon, Hong Kong, China}
\email{yangwang@ust.hk}

\baselineskip 18pt
\bibliographystyle{plain}
\title[Smoothed amplitude model]{The global landscape of phase retrieval \RNum{2}: quotient intensity models}
\maketitle

\begin{abstract}
A fundamental problem in phase retrieval is to reconstruct an unknown signal from a set of
 magnitude-only measurements.
In this work we introduce three novel quotient intensity-based models (QIMs) based a deep
modification of the traditional intensity-based models. A remarkable feature of  the new
 loss functions is that the corresponding geometric landscape  is benign under the optimal sampling complexity.  When the measurements $ a_i\in \Rn$ are Gaussian random vectors and the number of measurements $m\ge Cn$, the QIMs admit no spurious local minimizers with high probability, i.e., the target solution $ x$ is the unique global minimizer (up to a global phase) and the loss function has a negative directional curvature around each saddle point.   Such benign geometric landscape allows the gradient descent methods to find the global solution $x$ (up to a global phase) without spectral initialization.
\end{abstract}

\section{Introduction}
\subsection{Background}
The intensity-based model for phase retrieval is 
\[
y_j=\abs{\nj{a_j, x}}^2, ~j=1,\ldots,m,
\]
where $ a_j\in \Rn, j=1,\ldots,m$ are given vectors and $m$ is the number of measurements.  The phase retrieval problem aims to 
 recover the unknown signal $x\in \Rn$
based on the measurements $\dkh{(a_j,y_j)}_{j=1}^m$.  A natural approach to solve this problem is 
to consider the minimization problem
\begin{align} \label{fu_001}
\min_{u\in \R^n} \qquad f(u) = \frac 1 m \sum_{j=1}^m  { ((a_j \cdot u)^2 - (a_j \cdot x)^2 )^2}.
\end{align}
However, as shown in \cite{Sun18}, to guarantee the above loss function to have benign geometric landscape, the requirement of sampling complexity is $O(n \log^3 n)$.  This result is recently improved to $O(n \log n)$  in  \cite{2021b}. On the other hand, due to the heavy tail of the
quartic random variables in \eqref{fu_001}, such results seem to be optimal for this class of
loss functions.

To remedy this issue, we propose in this work three novel quotient intensity-based models (QIM)s to recover $x$ under optimal sampling complexity.
We rigorously prove that, for Gaussian random measurements,  those empirical loss functions admit the benign geometric landscapes with high probability under the optimal sampling complexity $O(n)$. Here, the phrase ``benign'' means: (1) the loss function has no spurious local minimizers;  and (2) the loss function has a negative directional curvature around each saddle point.  
The three quotient intensity-based models are 

QIM1:
\begin{align}  \label{mo:qwf1}
\min_{u\in \R^n} \qquad f(u) &= \frac 1 m \sum_{k=1}^m 
\frac{ ( (a_k\cdot u)^2 - (a_k\cdot x)^2 )^2}
{(a_k\cdot x)^2}.
\end{align}

QIM2:
\begin{align} \label{mo:qwf2}
\min_{u\in \R^n} \qquad f(u) &= \frac 1 m \sum_{k=1}^m 
\frac{ ( (a_k\cdot u)^2 - (a_k\cdot x)^2 )^2}
{\beta |u|^2 +(a_k\cdot x)^2}.
\end{align}

QIM3:
\begin{align} \label{mo:qwf3}
\min_{u\in \R^n} \qquad f(u) &= \frac 1 m \sum_{k=1}^m 
\frac{ ( (a_k\cdot u)^2 - (a_k\cdot x)^2 )^2}
{ |u|^2 +\beta_1(a_k\cdot u)^2+ \beta_2 (a_k\cdot x)^2}.
\end{align}

The phase retrieval problem arises in many fields of science and engineering such as X-ray crystallography \cite{harrison1993phase,millane1990phase}, microscopy
\cite{miao2008extending}, astronomy \cite{fienup1987phase}, coherent diffractive
imaging \cite{shechtman2015phase,gerchberg1972practical} and optics
\cite{walther1963question} etc.  In practical applications
due to the physical limitations  optical detectors  can only record the magnitude of signals while losing the phase information. 
Many algorithms have been designed to solve the phase retrieval problem, which includes convex algorithms and non-convex ones. The convex algorithms  usually rely on a ``matrix-lifting'' technique, which is computationally inefficient for large scale problems \cite{phaselift,Phaseliftn,Waldspurger2015} . 
 In contrast, many non-convex algorithms bypass the lifting step and operate directly on the lower-dimensional ambient space, making them much more computationally efficient.  Early non-convex algorithms were mostly based on the technique of alternating projections, e.g. Gerchberg-Saxton \cite{Gerchberg1972} and Fineup \cite{ER3}. The main drawback, however, is the lack of theoretical guarantee. Later Netrapalli et al \cite{AltMin} proposed the AltMinPhase  algorithm based on a technique known as {\em spectral initialization}. They proved that the algorithm linearly converges to the true solution with $O(n \log^3 n)$ resampling Gaussian random measurements. This  work led 
further to several other non-convex algorithms based on spectral initialization. 
A common thread  is 
 first choosing a good initial guess through spectral initialization, and then solving an optimization model through gradient descent, such as the Wirtinger Flow  method  \cite{WF},  Truncated Wirtinger Flow algorithm \cite{TWF},  randomized Kaczmarz method \cite{huang2021linear,tan2019phase,Wei2015},
 Gauss-Newton method \cite{Gaoxu}, Truncated Amplitude Flow algorithm \cite{TAF}, Reshaped Wirtinger Flow (RWF) \cite{RWF} and so on.

\subsection{Prior arts and connections} \label{sec:prob}
As was already mentioned earlier, producing a good initial guess using spectral initialization seems to
be a prerequisite for prototypical  non-convex algorithms to succeed with good theoretical guarantees. A natural and fundamental  question is:

{\em Is it possible for non-convex algorithms to achieve successful recovery with a random initialization (i.e. without spectral initialization or any additional truncation)}?  

  In the recent
work \cite{Sun18}, Ju Sun et al.  carried out a deep study of the global geometric structure of  phase retrieval problem. They proved that the loss function does not have any spurious local minima under $O(n \log^3 n)$ Gaussian random measurements. More specifically, it was shown in \cite{Sun18} that all 
minimizers coincide with the target signal $\vx$ up to a global phase, and the loss function has a negative directional curvature around each saddle point. Thanks to this benign geometric landscape any algorithm which can avoid saddle points  converges to the true solution with high probability. 
A trust-region
method was employed in \cite{Sun18} to find the global minimizers with random initialization. To reduce the sampling complexity, it has been shown in \cite{cai2019} that a  combination of the loss function with a judiciously chosen activation function also possesses the benign geometry structure  under $O(n)$ Gaussian random measurements. Recently, a smoothed amplitude flow estimator has been proposed in \cite{2020a} and the authors show that the loss function has benign geometry structure under the optimal sampling complexity. Numerical tests show that the estimator in \cite{2020a}  yields very stable and fast convergence with random initialization and performs as good as or even better than the existing gradient descent methods with spectral initialization.

The emerging concept of a benign geometric landscape has also recently been   explored in many other applications of signal processing and machine learning, e.g. matrix sensing \cite{bhojanapalli2016global,park2016non}, tensor decomposition \cite{ge2016matrix}, dictionary learning\cite{sun2016complete} and matrix completion \cite{ge2015escaping}. 
For general optimization problems there exist a plethora of loss functions with 
well-behaved geometric landscapes such that all local optima are also global optima and 
each saddle point has a negative direction curvature in its vincinity. 
Correspondingly several techniques have been developed to guarantee that the standard gradient based optimization algorithms can escape such saddle points efficiently, see e.g. \cite{jin2017escape,du2017gradient,jin2017accelerated}.

\subsection{Our contributions}
This paper aims to show the intensity-based model \eqref{fu_001} with some deep modification has a benign geometry structure  under the optimal sampling complexity. More specifically, we first  introduce three novel quotient intensity models
 and then we prove rigorously that each loss function of them has no spurious local minimizers. Furthermore,  the loss function of quotient intensity model has a negative directional curvature around each saddle point. Such properties allow first order method like gradient descent  to locate a global minimum with random initial guess.

Our first result shows that the loss function of \eqref{mo:qwf1} has the benign geometric landscape, as stated below. 

\begin{theorem}[Informal]
Consider the quotient intensity model \eqref{mo:qwf1}.
Assume  $\{\va_i\}_{i=1}^m$ are i.i.d.  standard Gaussian random vectors and $\vx\ne 0$. 
There exist positive absolute constants $c$, $C$, such that if $m\ge C n  $, then
with probability at least $1- e^{-cm} $ the loss function $F=F(\vz)$ 
has no spurious local minimizers. The only global minimizers are
$\pm  \vx$. All other critical points are strict saddles.
\end{theorem}

The second result is the global analysis for the estimator \eqref{mo:qwf2}.

\begin{theorem}[Informal]
Consider the quotient intensity model \eqref{mo:qwf2}.
Let $0<\beta <\infty$. Assume 
$\{\va_i\}_{i=1}^m$ are i.i.d.  standard Gaussian random vectors and $\vx\ne 0$. 
There exist positive constants $c$, $C$ depending only on $\beta$, such that if $m\ge C n  $, then
with probability at least $1- e^{-cm} $ the loss function $F=F(\vz)$ 
has no spurious local minimizers. The only global minimizer is
$\pm  \vx$ and all other critical points are strict saddles.
\end{theorem}

\begin{remark}
There appears some  subtle differences between estimators \eqref{mo:qwf1} and
\eqref{mo:qwf2}.  Although the former looks more singular, one can prove
full strong convexity in the neighborhood of the global minimizers. In
the latter case, however, we only have certain restricted convexity. 
\end{remark}

The third result is the global landscape for the estimator \eqref{mo:qwf3}.
\begin{theorem} [Informal]
Consider the quotient intensity model \eqref{mo:qwf3}.
Let $0<\beta_1,\beta_2 <\infty$. Assume 
$\{\va_i\}_{i=1}^m$ are i.i.d.  standard Gaussian random vectors and $\vx\ne 0$. 
There exist positive constants $c$, $C$ depending only on $\beta$, such that if $m\ge C n  $, then
with probability at least $1- e^{-cm} $ the loss function $F=F(\vz)$ 
has no spurious local minimizers. The only global minimizer is
$\pm  \vx$ and all other critical points are strict saddles.
\end{theorem}

\begin{remark}
For this case, thanks to the strong damping, we have
full strong convexity in the neighborhood of the global minimizers. 
\end{remark}

\subsection{Notations}
Throughout this proof we fix $\beta>0$ as a constant and do not study
the precise dependence of other parameters on $\beta$.  We write $u \in \mathbb S^{n-1}$ if $u\in \mathbb R^n$ and 
$\|u \|_2=\sqrt{ \sum_{j} u_j^2} =1$. 
 We use $\chi$ to denote the usual characteristic function. For example  $\chi_A (x)=1$ if $x \in A$ and $\chi_A(x)=0$ if $x\notin A$. 
 We denote by $\delta_1$, $\epsilon$, $\eta$, $\eta_1$ various 
constants whose value will be taken sufficiently small. The needed smallness will be
clear from the context. 
 For any quantity $X$, we shall write $X=O(Y)$ if $|X| \le C Y$ for some constant
$C>0$.  We write $X\lesssim Y$ if $X\le CY$ for some constant $C>0$.  
We shall write $X \ll Y$ if $ X \le c Y$ where the constant $c>0$ will be sufficiently
small. 
 In our proof it is important for us to specify the precise dependence of the sampling size $m$ in terms of the dimension $n$. For this purpose
we shall write $m\gtrsim n $ if $m\ge C n$ where the constant $C$ is allowed to depend on 
$\beta$ and 
the small constants $\epsilon$, $\epsilon_i$ etc used in the argument.  
One can extract more explicit dependence 
of $C$ on the small constants and $\beta$ but for simplicity we suppress this dependence here. 
 We shall say an event $A$ happens with \textbf{high probability} if   $\mathbb P (A) \ge   1-C e^{-cm}$, 
where $c>0$, $C>0$ are constants.  The constants $c$ and $C$ are allowed to depend
on $\beta$ and the small constants $\epsilon$, $\delta$ mentioned before.

\subsection{Organization}
In Section 2--4 we carry out an in-depth analysis of the corresponding geometric landscape of QIM1, QIM2 and QIM3 under optimal sampling complexity $O(n)$. 
In Section 5, we report some numerical experiments to  demonstrate the efficiency of our proposed estimators.
In Appendix, we collect the technique lemmas which are used in the proof.

\section{Quotient intensity model \RNum{1}}  \label{S:model4a}
In this section, we consider the first quotient intensity model and prove that it has benign geometric landscape, as demonstrated below.

\begin{align} \label{model4ae1}
f(u) &= \frac 1 m \sum_{k=1}^m 
\frac{ ( (a_k\cdot u)^2 - (a_k\cdot x)^2 )^2}
{(a_k\cdot x)^2}.
\end{align}

\begin{thm} \label{thmEa}
Assume 
$\{a_k\}_{k=1}^m$ are i.i.d.  standard Gaussian random vectors and $x\ne 0$. 
There exist positive absolute constants $c$, $C$, such that if $m\ge C n  $, then
with probability at least $1- e^{-cm} $ the loss function $f=f(u)$ 
defined by \eqref{model4ae1}
has no spurious local minimizers. The only global minimizer is
$\pm  x$, and the loss function is strongly convex in a neighborhood of $\pm x$.
The point $u=0$ is a local maximum point with strictly negative-definite Hessian.
All other critical points are strict saddles, i.e., each saddle point has a neighborhood
where the function has negative directional curvature.
\end{thm}

Without loss of generality we shall assume $x=e_1$ throughout the rest of the proof.
Note that the set $\bigcup_{k=1}^m \{ a_k \cdot e_1=0 \}$ has measure zero.
Thus for typical realization we have $a_k\cdot e_1 \ne 0$ for all $k$. This means
that the loss function $f(u)$ defined by \eqref{model4ae1} is smooth almost
surely. 
We denote the Hessian of the function $f(u)$ along the 
$\xi$-direction ($\xi\in \mathbb S^{n-1}$) as
\begin{align}
H_{\xi\xi}(u)
& = \sum_{i,j=1}^n \xi_i \xi_j (\partial_{ij} f)(u) =  \frac 4 m \sum_{k=1}^m
\Bigl( 3 \frac {(a_k\cdot \xi)^2 (a_k\cdot u)^2} {(a_k\cdot e_1)^2}
-(a_k\cdot \xi)^2 \Bigr). \label{Sep7_Hessian}
\end{align}
\subsection{Strong convexity near the global minimizers
$u=\pm e_1$}
\begin{thm}[Strong convexity near $u=\pm e_1$] \label{thm_Sep7_00}
There exists an absolute constant $0<\epsilon_0\ll 1$ such that
the following hold. For $m\gtrsim n$, it holds with high probability that 
\begin{align*}
H_{\xi\xi}(u) \ge 1, \qquad \forall\, \xi \in \mathbb S^{n-1}, \quad
\forall\, \text{$u$ with $\|u \pm e_1\|_2 \le \epsilon_0$}.
\end{align*}
\end{thm}
\begin{proof}
By Lemma \ref{lemSe7_0a}, we can take $\epsilon>0$ sufficiently small,
$N$ sufficiently large such that 
\begin{align*}
\mathbb E \frac {(a_k\cdot \xi)^2 (a_k\cdot e_1)^2}
{\epsilon+ (a_k\cdot e_1)^2}
\phi(\frac {a_k\cdot \xi} N)
\ge 0.99, \qquad\forall\, \xi \in \mathbb S^{n-1}, \;\forall\, 1\le k\le m.
\end{align*}
In the above $\phi \in C_c^{\infty}(\mathbb R)$ satisfies
$0\le \phi(x) \le 1$ for all $x$, $\phi(x)=1$ for $|x| \le 1$
and $\phi(x)=0$ for $|x|\ge 2$. 
Clearly  if $\|u\pm e_1\|_2\le \epsilon_0$ and $\epsilon_0$ 
is sufficiently small (depending on $\epsilon$ and $N$), then
\begin{align*}
\mathbb E \frac {(a_k\cdot \xi)^2 (a_k\cdot u)^2}
{\epsilon+ (a_k\cdot e_1)^2}
\phi(\frac {a_k\cdot \xi} N)
\ge 0.98, \qquad\forall\, \xi \in \mathbb S^{n-1}, \;\forall\, 1\le k\le m.
\end{align*}
The above term inside the expectation is clearly OK for union bounds. 
Thus for $\|u\pm e_1\|\le \epsilon_0$ and $m\gtrsim n$, it holds with high probability that 
\begin{align*}
\frac 14 H_{\xi\xi} (u)
&\ge   \frac 1m \sum_{k=1}^m \Bigl(
\frac {(a_k\cdot \xi)^2 (a_k\cdot u)^2}
{\epsilon+ (a_k\cdot e_1)^2}
\phi(\frac {a_k\cdot \xi} N) -  (a_k\cdot \xi)^2 \Bigr)  \ge 3\cdot 0.97 - 1.01, \qquad\forall\, \xi \in \mathbb S^{n-1}.
\end{align*}
Thus the desired inequality follows.
\end{proof}

\subsection{The regimes $\|u\|_2\ll 1 $ and $\|u\|_2\gg 1$ are fine}$\;$

We first investigate the point $u=0$. It is trivial to verify that $\nabla f(0)=0$ since 
$a_k\cdot e_1 \ne 0$ for all $k$ almost surely. 
\begin{lem}[$u=0$ has strictly negative-definite Hessian] \label{Sep7_e0}
We have $u=0$ is a local maximum point with strictly
negative-definite Hessian. More precisely, for $m\gtrsim n$, it holds with high probability that
\begin{align*}
\sum_{k,l=1}^n \xi_k \xi_l (\partial_{kl} f)(0)
\le -1, \quad\forall\, \xi \in \mathbb S^{n-1}.
\end{align*}
\end{lem}
\begin{proof}
By \eqref{Sep7_Hessian}, it is obvious that 
\begin{align*}
 H_{\xi\xi}(0) = -4 \frac 1m \sum_{k=1}^m (a_k\cdot \xi)^2.
\end{align*}
The desired conclusion then easily follows from Bernstein's inequality.
\end{proof}

Write  $u=\sqrt{R} \hat u$ where $\hat u \in S^{n-1}$ and $R>0$.   Then
\begin{align*}
f (u) = \frac 1m \sum_{k=1}^m 
\frac { \Bigl( R (a_k\cdot \hat u)^2 -(a_k\cdot e_1)^2 \Bigr)^2}
{(a_k\cdot e_1)^2}.
\end{align*}
A simple calculation leads to 
\begin{align}
&\partial_{R} f =   
2R \frac 1m \sum_{k=1}^m \frac {(a_k\cdot \hat u)^4}{(a_k\cdot e_1)^2}
- 2 \frac 1m \sum_{k=1}^m (a_k\cdot \hat u)^2;
   \label{Sep7e10a}\\
&\partial_{RR} f =2\frac 1m \sum_{k=1}^m
\frac{ (a_k\cdot \hat u)^4}
{  (a_k\cdot e_1)^2}. \label{Sep7e10b}
\end{align}

\begin{lem}[The regime $\|u\|_2\ge 1+\epsilon_0$ is OK] \label{Sep7_e1}
Let $0<\epsilon_0\ll 1$ be any given small constant.  Then
the following hold: For $m\gtrsim n$,  with high probability it holds that
\begin{align*}
\partial_{R} f >0,  \quad\forall\, R \ge 1+\epsilon_0 , \;\text{}\forall\, \hat u \in \mathbb S^{n-1}.
\end{align*}
\end{lem}
\begin{proof}
 Denote $X_k =a_k\cdot e_1$ and $Z_k= a_k \cdot \hat u$. 
By \eqref{Sep7e10a} and Cauchy-Schwartz, 
 we have 
\begin{align*}
\partial_R f & \ge 
 \frac {2R}m  \frac { \Bigl(\sum_{k=1}^m (a_k\cdot \hat u)^2
\Bigr)^2 } {
\sum_{k=1}^m (a_k\cdot e_1)^2 } - \frac 2m \sum_{k=1}^m
(a_k\cdot \hat u)^2 \notag \\
& \ge \;2R \cdot (1-\delta_1) -2 (1+\delta_1),
\qquad\forall\, \hat u \in \mathbb S^{n-1},
\end{align*}
where  $0<\delta_1 \ll 1$ is an absolute constant which we can take
to be sufficiently small, and in the last inequality
we have used Bernstein.  The desired result
then easily follows by taking $R\ge R_1=\frac {1+2\delta_1} {1-\delta_1}$ and
choosing $\delta_1$ such that $R_1\le 1+\epsilon_0$.
\end{proof}

From \eqref{Sep7e10a},  due to the highly irregular coefficients near $R$, it is 
difficult to control the upper bound of $\partial_R f$ in the regime $R\ll 1$. 
To resolve this difficulty, we shall examine the Hessian in this regime.

\begin{lem}[The regime $\|u\|_2 \le \frac 13 $ is OK] \label{Sep7_e2}
For $m\gtrsim n$,  with high probability it holds that
\begin{align*}
 H_{e_1e_1} (u)  \le - \frac 12 <0 ,  \quad\forall \quad  0<\|u\|_2 \le \frac 13,
\end{align*}
where $H_{e_1e_1}$ is defined in \eqref{Sep7_Hessian}.
\end{lem}
\begin{proof}
It follows from \eqref{Sep7_Hessian} together with Bernstein's inequality that for $m\gtrsim n$ with high probability, it holds
\begin{align*}
\frac 14 H_{e_1e_1}(u)
&= \frac 1 m\sum_{k=1}^m 
\Bigl( 3 (a_k\cdot u)^2 -(a_k\cdot e_1)^2 \Bigr) \le \|u\|_2^2 \cdot 3 \cdot \frac {10}9 
- \frac 89 \le -\frac 12.
\end{align*}
This completes the proof.
\end{proof}

\begin{thm}[The regimes $\|u\|_2\le \frac 13$
and $\|u\|_2\ge 1+\epsilon_0$ are OK] \label{thmSep7_1}
Let $0<\epsilon_0\ll 1$ be a given small constant.
For $m\gtrsim n$, with high probability the following hold:
\begin{enumerate}
\item We have 
\begin{align*}
&\partial_{R} f >0 , \qquad\forall\, R\ge 1+\epsilon_0, 
\quad \forall\, \hat u \in \mathbb S^{n-1}.
\end{align*}
\item The point $u=0$ is a local maximum point with strictly negative-definite
Hessian, 
\begin{align*}
\sum_{k,l=1}^n  \xi_k \xi_l (\partial_{kl} f)(0) \le -1,
\qquad \forall\, \xi \in \mathbb S^{n-1}. 
\end{align*}
\item We have 
\begin{align*}
H_{e_1e_1}(u) \le -1, \qquad \forall\, \|u\|_2 \le \frac 13.
\end{align*}
\end{enumerate}
\end{thm}
\begin{proof}
This follows from Lemma \ref{Sep7_e0}, \ref{Sep7_e1} and \ref{Sep7_e2}.
\end{proof}

\begin{thm}[The regime $\|u\|_2 \sim 1 $, 
$||\hat u\cdot e_1|-1|\ge \eta_0$]\label{Sep7_e3}
 Let $0<\eta_0\ll 1$ be given. Then for $m\gtrsim n$, the following hold
 with high probability:
 Suppose $u =\sqrt R \hat u$, $ 1/9 \le R \le 2$, and  $ \Bigl | |\hat u \cdot e_1 | -1 \Bigr | \ge \eta_0$.
 If $(\partial_R f)(u) =0$, then we must have
 \begin{align*}
 H_{e_1 e_1} (u) <0.
 \end{align*}
 \end{thm}
 \begin{proof}
By \eqref{Sep7e10a}, we have if $\partial_R f (u)=0$, then 
\begin{align*}
R \frac 1m \sum_{k=1}^m \frac {(a_k\cdot \hat u)^4}{(a_k\cdot e_1)^2}
=  \frac 1m \sum_{k=1}^m (a_k\cdot \hat u)^2.
\end{align*}
By Lemma \ref{lemSe7_1a},  we have for $m\gtrsim n$, it holds
 with high probability that
 \begin{align*}
\frac 1m \sum_{k=1}^m \frac {(a_k\cdot \hat u)^4}{(a_k\cdot e_1)^2}
\ge 100, \quad\forall\, \hat u \in \mathbb S^{n-1} 
\text{ with $||\hat u \cdot e_1|-1| \ge \eta_0$}.
\end{align*}
Clearly then $R \le \frac 1 {50}$ with high probability. Thus it follows
easily that $H_{e_1e_1} (u)<0$ also with high probability. 
\end{proof}

\begin{thm}[Localization of $R$ when $||\hat u\cdot e_1|-1| \le \eta_0$, $R\le 1+\eta_0$ 
and $u$
is a critical point] \label{Sep7_e3bb}
Let $0<\eta_0\ll 1$ be given. For $m\gtrsim n$, the following hold with high probability:
Assume $u =\sqrt R \hat u$ is a critical point with
 $\frac 19 \le R \le 1+\eta_0$, and $ || \hat u\cdot e_1| -1| \le \eta_0$.  Then we must have
 \begin{align*}
 |R-1 | \le c(\eta_0),
 \end{align*}
 where $c(\eta_0) \to 0$ as $\eta_0\to 0$. 
\end{thm}
\begin{proof}
Denote $\partial_{\xi} f = \xi \cdot \nabla f $ for $\xi \in \mathbb S^{n-1}$. It is not
difficult to check that
\begin{align*}
\frac 1 4 \partial_{\xi} f
= \frac 1m \sum_{k=1}^m \frac {(a_k\cdot u)^3 (a_k\cdot \xi)}
{X_k^2}  - \frac 1m \sum_{k=1}^m (a_k\cdot u) (a_k\cdot \xi)=0,
\end{align*}
where $X_k = a_k \cdot e_1$.  
Setting $\xi = \hat u$ and $\xi =e_1$ respectively give us two equations:
\begin{align}
&R \cdot \Bigl( \frac 1m \sum_{k=1}^m 
\frac {(a_k\cdot \hat u)^4} {X_k^2} \Bigr) -
\frac 1m \sum_{k=1}^m (a_k\cdot \hat u)^2 =0, \label{Sep7_e3bb.01}\\
& R \cdot \Bigl( \frac 1m \sum_{k=1}^m
\frac {(a_k\cdot \hat u)^3} {X_k } \Bigr)
- \frac 1m \sum_{k=1}^m (a_k\cdot \hat u) X_k =0.
\end{align}
We then obtain 
\begin{align} \label{Sep7_e3bb.02}
\Bigl( \frac 1m \sum_{k=1}^m (a_k\cdot \hat u)^2 \Bigr)
\cdot \Bigl( \frac 1m \sum_{k=1}^m
\frac {(a_k\cdot \hat u)^3} {X_k}
\Bigr)
= \Bigl
( \frac 1m \sum_{k=1}^m \frac {(a_k\cdot \hat u)^4} {X_k^2}
\Bigr) \cdot \Bigl( \frac 1m \sum_{k=1}^m (a_k\cdot \hat u) X_k \Bigr).
\end{align}
Without loss of generality we assume $\| \hat u -e_1\|_2 \le \eta \ll 1$.  Then
with high probability we have
\begin{align*}
&  \frac 1m \sum_{k=1}^m (a_k\cdot \hat u) X_k  = 1 +O(\eta), \\
& \frac 1m \sum_{k=1}^m (a_k\cdot \hat u)^2 = 1+ O(\eta).
\end{align*}
Observe that by Cauchy-Schwartz,
\begin{align*}
\sum_{k=1}^m \frac {|a_k\cdot \hat u|^3} {|X_k|}
\le \Bigl( \sum_{k=1}^m \frac {|a_k\cdot \hat u|^4} {X_k^2} \Bigr)^{\frac 12}
\cdot \Bigl( \sum_{k=1}^m (a_k\cdot \hat u)^2 \Bigr)^{\frac 12}.
\end{align*}
Plugging the above estimates into \eqref{Sep7_e3bb.02}, we obtain
\begin{align*}
\sqrt{\frac 1m \sum_{k=1}^m \frac {(a_k\cdot \hat u)^4} {X_k^2}} 
\le 1 +O(\eta).
\end{align*}
Using \eqref{Sep7_e3bb.01}, we then get
\begin{align*}
R \ge  1+O(\eta).
\end{align*}
The desired result then easily follows.
\end{proof}
We now complete the proof of the main theorem.

\begin{proof}[Proof of Theorem \ref{thmEa}]
We proceed in several steps.
\begin{enumerate}
\item By Theorem \ref{thm_Sep7_00}, the function $f(u)$ is strongly convex
when $\| u\pm e_1\|_2 \ll 1$.

\item By Theorem \ref{thmSep7_1}, $f$ has non-vanishing gradient when $R\ge 1+\epsilon_0$.
Also $H_{e_1e_1}(u) \le -1$ when $\|u \|_2 \le \frac 13$. The point $u=0$
is a strict local maximum point with strictly negative-definite Hessian.

\item By Theorem \ref{Sep7_e3}, we have $H_{e_1e_1}(u)<0$ if $\|u\|_2 \sim 1$
and $| |\hat u \cdot e_1| -1| \ge \epsilon_0$.

\item Theorem \ref{Sep7_e3bb} shows that if $R\le 1+\epsilon_0$,
$| |\hat u\cdot e_1| -1| \le \epsilon_0$ and $u$ is a critical point, then
we must have $ |R-1| \le c(\epsilon_0) \ll 1$.  In yet other words we must
have $\| u\pm e_1\|_2 \ll 1$.  This regime is then treated by Step 1.

\end{enumerate}
\end{proof}

\section{Quotient intensity model \RNum{2}}  \label{S:model4b}
Consider for $\beta>0$, 
\begin{align} \label{model4be1}
f(u) &= \frac 1 m \sum_{k=1}^m 
\frac{ ( (a_k\cdot u)^2 - (a_k\cdot x)^2 )^2}
{\beta |u|^2 +(a_k\cdot x)^2}.
\end{align}

\begin{thm} \label{thmEb}
Let $0<\beta <\infty$. 
Assume 
$\{a_k\}_{k=1}^m$ are i.i.d.  standard Gaussian random vectors and $x\ne 0$. 
There exist positive constants $c$, $C$ depending only on $\beta$, such that if $m\ge C n  $, then
with probability at least $1- e^{-cm} $ the loss function $f=f(u)$ 
defined by \eqref{model4be1}
has no spurious local minimizers. The only global minimizer is
$\pm  x$, and the loss function is restrictively convex in a neighborhood of $\pm x$.
The point $u=0$ is a local maximum point with strictly negative-definite Hessian.
All other critical points are strict saddles, i.e., each saddle point has a neighborhood
where the function has negative directional curvature.
\end{thm}
\begin{rem}
See Theorem \ref{Sep19e8} for the precise statement concerning
restrictive convexity.
\end{rem}

Without loss of generality we shall assume $x=e_1$ throughout the rest of the proof.

\subsection{The regimes $\|u\|_2\ll 1 $ and $\|u\|_2\gg 1$ are fine}$\;$

We first investigate the point $u=0$. It is trivial to verify that $\nabla f(0)=0$ since 
$a_k\cdot e_1 \ne 0$ for all $k$ almost surely. 
\begin{lem}[$u=0$ has strictly negative-definite Hessian] \label{Sep8_e0}
We have $u=0$ is local maximum point with strictly
negative-definite Hessian. More precisely, 
 for $m\gtrsim n$, it holds with high probability that
\begin{align*}
\sum_{k,l=1}^n \xi_k \xi_l (\partial_{kl} f)(0)
\le -d_1, \quad\forall\, \xi \in \mathbb S^{n-1},
\end{align*}
where $d_1>0$ is an absolute constant.
\end{lem}
\begin{proof}
We begin by noting that  since almost surely $a_k\cdot e_1\ne 0$ for all $k$,
the function $f$ is smooth at $u=0$.  It suffices for us to consider 
(write $u=\sqrt{t} \xi$)
\begin{align*}
G(t) = \frac 1 m
\sum_{k=1}^m \frac { (t(a_k\cdot \xi)^2 - (a_k\cdot e_1)^2)^2}
{\beta t+ (a_k\cdot e_1)^2}.
\end{align*}
Clearly
\begin{align*}
G^{\prime}(0) =-\beta -2 \frac 1m \sum_{k=1}^m (a_k\cdot \xi)^2.
\end{align*}
The desired conclusion then easily follows by using Bernstein's inequality.
\end{proof}

Write  $u=\sqrt{R} \hat u$ where $\hat u \in S^{n-1}$ and $R>0$.   Then
\begin{align*}
f (u) = \frac 1m \sum_{k=1}^m 
\frac { \Bigl( R (a_k\cdot \hat u)^2 -(a_k\cdot e_1)^2 \Bigr)^2}
{\beta R + (a_k\cdot e_1)^2}.
\end{align*}
Clearly
\begin{align}
&\partial_{R} f =   
\frac 1m \sum_{k=1}^m
\frac 
{ R^2 (\beta (a_k\cdot \hat u)^4)
+2 R (a_k\cdot \hat u)^4 (a_k\cdot e_1)^2
-\beta (a_k\cdot e_1)^4 -2(a_k\cdot e_1)^4 (a_k\cdot \hat u)^2}
{
(\beta R+ (a_k\cdot e_1)^2 )^2};
   \label{Sep8e10a}\\
&\partial_{RR} f =2\frac 1m \sum_{k=1}^m
\frac{ (a_k\cdot e_1)^4 (\beta+ (a_k\cdot \hat u)^2)^2}
{ (\beta R+ (a_k\cdot e_1)^2)^3}. \label{Sep8e10b}
\end{align}

\begin{lem}[The regime $\|u\|_2\gg 1$ is OK] \label{Sep8_e1}
There exist constants $R_1=R_1(\beta)>0$, $d_1=d_1(\beta)>0$ such that the 
following hold: 
For $m\gtrsim n$,  with high probability it holds that
\begin{align*}
\partial_{R} f \ge d_1,  \quad\forall\, R \ge R_1 , \;\text{}\forall\, \hat u \in \mathbb S^{n-1}.
\end{align*}
\end{lem}
\begin{proof}
We only sketch the proof.  Denote $X_k =a_k\cdot e_1$ and $Z_k= a_k \cdot \hat u$. 
Using the inequalities (assume $R\gg 1$ and denote by $C_1>0$ a constant depending
only on $\beta$)
\[
\beta R+ X_k^2 \le  R (\beta +X_k^2), \qquad (\beta R +X_k^2)^2 \ge 4 \beta R X_k^2
\]
and 
\[
\frac {X_k^4} {(\beta R+ X_k^2)^2} \le C_1\cdot ( \frac {R} {R^2} +   \chi_{|X_k|\ge R^{{\frac 14} } }),
\]
we have
\begin{align*}
\partial_R f
&\ge \frac 1m \sum_{k=1}^m
\frac {\beta Z_k^4} {(\beta +X_k^2)^2}
\phi(\frac {Z_k} N) 
- \frac 1m \sum_{k=1}^m \frac  1 {4R} X_k^2
-\frac 2 m \sum_{k=1}^m C_1 \cdot (R^{-1} +\chi_{|X_k|\ge R^{\frac 14} } )
\cdot Z_k^2,
\end{align*}
where we have chosen $\phi \in C_c^{\infty}$ such that $0\le \phi(x)\le 1$
for all $x$, $\phi(x)=1$ for $|x| \le 1$ and $\phi(x)=0$ for $|x| \ge 2$. 
Observe that for $a\sim \mathcal N (0, \operatorname{I_n} )$, $Z \sim \mathcal N(0,1)$,
\begin{align*}
\mathbb E (a\cdot \hat u)^4 \chi_{|a\cdot \hat u|\ge N} 
\le \mathbb E Z^4 \chi_{|Z|\ge N} \to 0, \qquad \text{as $N\to \infty$}.
\end{align*}
It is also easy to show that 
\begin{align*}
\inf_{\hat u \in \mathbb S^{n-1}}
\mathbb E \frac {(a\cdot \hat u)^4} {(\beta +(a\cdot e_1)^2)^2} \gtrsim 1.
\end{align*}
Thus we can take $N$ large such that
\begin{align*}
\inf_{\hat u \in \mathbb S^{n-1} }\mathbb E \frac { (a\cdot \hat u)^4} {(\beta +(a\cdot e_1)^2)^2}
\phi( \frac {a\cdot \hat u} N) \gtrsim 1.
\end{align*}
It is easy to show that by taking $R$ large, for $m\gtrsim n$, it holds with high
probability that
\begin{align*}
\frac 1m \sum_{k=1}^m \chi_{|X_k| \ge R^{\frac 14} } Z_k^2 \le \epsilon.
\end{align*}
Since all the other terms are OK for union bounds,
the desired result then clearly follows by taking $R$ large.
\end{proof}

\begin{lem}[The regime $\|u\|_2 \ll 1 $ with
$ \frac {|u_1|} {\|u\|_2} \le \frac 1 {10}$ is OK] \label{Sep8_e2}
There exist a constant $R_2=R_2(\beta)>0$  such that the 
following hold: 
For $m\gtrsim n$,  with high probability it holds that
\begin{align*}
\partial_{u_1 u_1} f \le -2<0 ,  \quad\forall\, 0<R \le R_2 , \;\text{}\forall\, \hat u \in \mathbb S^{n-1} \text{with $  |\hat u \cdot e_1| \le \frac 1 {10} $}. 
\end{align*}
\end{lem}
\begin{proof}
We only sketch the proof.  Denote $X_k=a_k\cdot e_1$ and $Z_k= a_k \cdot \hat u$. 
A short computation gives
\begin{align*}
\partial_{u_1u_1} f 
&=  \frac 4m \sum_{k=1}^m \frac {3 R X_k^2 Z_k^2 - X_k^4} {\beta R + X_k^2}+ \frac 1m \sum_{k=1}^m ( R Z_k^2-X_k^2)^2
\cdot \frac { 6\beta^2 u_1^2 -2\beta^2 |u^{\prime}|^2-2\beta X_k^2 } { (\beta R+X_k^2)^3} \notag \\
&\quad -16\beta (\hat u \cdot e_1) R^2 \frac 1m \sum_{k=1}^m
\frac { Z_k^3 X_k } {(\beta R+ X_k^2)^2} 
+16\beta (\hat u\cdot e_1) R \frac 1m \sum_{k=1}^m \frac{ X_k^3 Z_k} { (\beta R+ X_k^2)^2},
\end{align*}
where $u_1= u\cdot e_1$ and $u^{\prime}= u - u_1 e_1$. 
Now observe that 
\begin{align*}
& \frac 1m \sum_{k=1}^m\frac {Z_k^2} {\beta R+ X_k^2} X_k^2 
\le \frac 1m \sum_{k=1}^m Z_k^2; \\
& \frac 1m \sum_{k=1}^m \frac {X_k^4}
{\beta R+ X_k^2} = 
\frac 1m \sum_{k=1}^m \frac { (\beta R+ X_k^2-\beta R)^2}
{\beta R+X_k^2}  
\ge \Bigl( \frac 1m \sum_{k=1}^m (\beta R+X_k^2)\Bigr) - 2\beta R
\ge -\beta R + \frac 1m \sum_{k=1}^m X_k^2 ; \\
& \frac 1m \sum_{k=1}^m \frac {R^{\frac 32} |Z_k|^3 R^{\frac 12} |X_k|}
{(\beta R+X_k^2)^2}
\le \epsilon_1 \frac 1m \sum_{k=1}^m \frac { R^2 Z_k^4}
{ (\beta R+X_k^2)^2} + \frac 1 {\epsilon_1^3}
\frac 1m \sum_{k=1}^m \frac {R^2X_k^4}{ (\beta R+X_k^2)^2} 
\le  \frac {R^2} {\epsilon_1^3} +\epsilon_1 \frac 1m \sum_{k=1}^m \frac { R^2 Z_k^4}
{ (\beta R+X_k^2)^2} ; \\
& \frac R  m \sum_{k=1}^m \frac {|X_k|^3 |Z_k|} {(\beta R+X_k^2)^2}
\le   \frac R  m\sum_{k=1}^m
\frac {|X_k|^3 |Z_k|}{ ( 3 (\beta R)^{\frac 13} ( \frac 14 X_k^4)^{\frac 13})^2} 
 \lesssim R^{\frac 13} \beta^{-\frac 23} \frac 1m \sum_{k=1}^m
|X_k|^{\frac 13} |Z_k|
\lesssim  R^{\frac 13} \beta^{-\frac 23}
\frac 1 m\sum_{k=1}^m (X_k^2 +Z_k^2+1),
\end{align*}
where in the above the constant $\epsilon_1>0$ will be taken sufficiently small. The needed
smallness will become clear momentarily.
Since $|u_1|/\|u\|_2 \le \frac 1 {10}$, it is clear that for some absolute constant $C_1>0$, 
\begin{align*}
\frac { 6\beta^2 u_1^2 -2\beta^2 |u^{\prime}|^2-2\beta X_k^2 } { (\beta R+X_k^2)^3}
\le - \beta  C_1 \cdot \frac 1 { (\beta R +X_k^2)^2}.
\end{align*}
Now
\begin{align*}
- \frac 1m \sum_{k=1}^m \frac {(RZ_k^2-X_k^2)^2}
{(\beta R+X_k^2)^2} \le
-\frac 1 m \sum_{k=1}^m \frac {R^2 Z_k^4}
{ (\beta R+X_k^2)^2} +  \frac {2R}m \sum_{k=1}^m \frac {Z_k^2} {\beta R+ X_k^2}. 
\end{align*}
Now take $\epsilon_1 = \frac {C_1}{1000}$.  By Lemma \ref{lemSeH:0}, we can take
$R$ sufficiently small such that with high probability
\begin{align*}
 \frac {2R}m \sum_{k=1}^m \frac {Z_k^2} {\beta R+ X_k^2} <\frac 1 {100}.
\end{align*}
All the other terms can be treated by taking $R$ sufficiently small, and the desired result
follows easily.
\end{proof}

\begin{lem}[The regime $\|u\|_2 \ll 1 $ with
$ \frac {|u_1|} {\|u\|_2} > \frac 1 {10}$ is OK] \label{Sep8_e2aa}
There exist a constant $R_3=R_3(\beta)>0$  such that the 
following hold: 
For $m\gtrsim n$, with high probability it holds that the loss function
$f=f(u)$ has no critical points in the regime 
\begin{align*}
\Bigl\{ u = \sqrt {R} \hat u: \quad 0<R \le R_3, \, \hat u \in \mathbb S^{n-1}
\text{ and } |\hat u \cdot e_1|>\frac 1{10} 
\Bigr\}.
\end{align*}
\end{lem}
\begin{proof}
We assume that for $0<R \ll 1$ there exists some critical point. 
The idea is to 
examine the necessary conditions for a potential critical point and then derive
a lower bound on $R$.
Denote $X_k=a_k\cdot e_1$ and $Z_k=a_k\cdot \hat u$. By 
\eqref{Sep8e10a}, we have $\partial_R f =0$ which gives
\begin{align*}
R \underbrace{\frac 1 m  \sum_{k=1}^m\frac { R \beta Z_k^4
+2 Z_k^4 X_k^2} 
{ (\beta R+ X_k^2)^2}}_{=: A_1}
= \underbrace{\frac 1m \sum_{k=1}^m
\frac { \beta X_k^4 + 2 X_k^4  Z_k^2}
{(\beta R +X_k^2 )^2}}_{=:B_1}.
\end{align*}
On the other hand, by using $\partial_{u_1} f(u) =0$, we obtain
\begin{align*}
\beta u_1 R^2 
\frac 1m \sum_{k=1}^m
\frac {Z_k^4} {(\beta R+X_k^2)^2}
-2 R^{\frac 32}
\frac 1m \sum_{k=1}^m
\frac {Z_k^3 X_k} {\beta R+ X_k^2}
=\beta u_1 \frac 1m \sum_{k=1}^m
\frac {2R Z_k^2 X_k^2-X_k^4}
{(\beta R+X_k^2)^2}
-2R^{\frac 12}
\frac 1m \sum_{k=1}^m 
\frac {Z_k X_k^3} {\beta R+ X_k^2}.
\end{align*}
Thus
\begin{align*}
-R \underbrace{ \widehat u_1 
\frac 1m \sum_{k=1}^m
\frac {\beta R Z_k^4} {(\beta R+X_k^2)^2}}_{=:A_2}
+R \cdot \underbrace{2 
\frac 1m \sum_{k=1}^m
\frac {Z_k^3 X_k} {\beta R+ X_k^2}}_{=:A_3}
=\underbrace{\beta \widehat u_1 \frac 1m \sum_{k=1}^m
\frac {-2R Z_k^2 X_k^2+X_k^4}
{(\beta R+X_k^2)^2}
+2
\frac 1m \sum_{k=1}^m 
\frac {Z_k X_k^3} {\beta R+ X_k^2}}_{=:B_2}.
\end{align*}
Without loss of generality we assume $\widehat{u}_1>\frac 1 {10}$.  Observe 
that for $0<R \le 1$, we have
\begin{align*}
\frac 1m \sum_{k=1}^m
\frac {X_k^4} {(\beta +X_k^2)^2}
\lesssim B_1 \lesssim 1+\frac 1m \sum_{k=1}^m Z_k^2.
\end{align*}
Thus with high probability $B_1 \sim 1$. 

Now by Lemma \ref{lemSeH:0}, for $0<R\ll 1$, we have
\begin{align*}
\frac 1m \sum_{k=1}^m
\frac {R Z_k^2 X_k^2} {(\beta R+ X_k^2)^2}
\le \frac 1m \sum_{k=1}^m \frac { R Z_k^2} { \beta R +X_k^2}
\ll 1.
\end{align*}
Also for $0<R\le 1$, we have
\begin{align*}
\frac 1m \sum_{k=1}^m
\frac {X_k^4} {(\beta +X_k^2)^2}
\le \frac 1m \sum_{k=1}^m
\frac {X_k^4}{ (\beta R+ X_k^2)^2}
\le 1.
\end{align*}
By Lemma \ref{lemSeH:1}, for $0<R \ll 1$, we have
\begin{align*}
c_1\le \frac 1m \sum_{k=1}^m \frac {Z_k X_k^3}
{\beta R+ X_k^2} \le c_2, 
\end{align*}
where $c_1$, $c_2>0$ are constants depending only on $\beta$. 
Thus with high probability we have for $0<R \ll 1$, $B_2 \sim 1$. 

Now since 
\begin{align*}
R A_1 = B_1, \qquad -R A_2 +R A_3=B_2;
\end{align*}
we obtain
\begin{align*}
A_1 + B_3  A_2 = B_3 A_3,
\end{align*}
where $B_3 = B_1 /B_2$.  Observe that $A_1>0$, $A_2>0$, and
\begin{align*}
& A_1 \sim  \frac 1m \sum_{k=1}^m \frac {Z_k^4} {\beta R + X_k^2}; \\
& A_3 \le \Bigl(
\frac 1m \sum_{k=1}^m 
\frac {Z_k^4} {\beta R+ X_k^2} \Bigr)^{\frac 34}
\Bigl( \frac 1m \sum_{k=1}^m \frac {X_k^4}{\beta R+ X_k^2}
\Bigr)^{\frac 14}. 
\end{align*}
It follows easily that with high probability we have
\begin{align*}
A_1 \sim 1.
\end{align*}
But then it follows from the equation $R A_1 =B_1$ that we must have
$R \sim 1$.  Thus the desired result follows.
\end{proof}

\begin{thm}[The regimes $\|u\|_2\ll 1$
and $\|u\|_2\gg 1$ are OK] \label{thmSep8_1}
For $m\gtrsim n$, with high probability the following hold:
\begin{enumerate}
\item We have 
\begin{align*}
\partial_R f \ge d_1, \qquad \forall\, R \ge R_1, \, \forall\, \hat u
\in \mathbb S^{n-1},
\end{align*}
where $d_1$,  $R_1$ are constants depending only on $\beta$. 
\item We have 
\begin{align*}
\partial_{u_1 u_1} f \le -2<0 ,  \quad\forall\, 0<R \le R_2 , \;\text{}\forall\, \hat u \in \mathbb S^{n-1} \text{with $  |\hat u \cdot e_1| \le \frac 1 {10} $},
\end{align*}
where $R_2>0$ is a constant depending only on $\beta$.
\item The loss function $f=f(u)$ has no critical points in the regime 
\begin{align*}
\Bigl\{ u = \sqrt {R} \hat u: \quad 0<R \le R_3, \, \hat u \in \mathbb S^{n-1}
\text{ and } |\hat u \cdot e_1|>\frac 1{10} 
\Bigr\},
\end{align*}
where $R_3>0$ is a constant depending only on $\beta$.

\item The point $u=0$ is a local maximum point with strictly negative-definite
Hessian, 
\begin{align*}
\sum_{k,l=1}^n  \xi_k \xi_l (\partial_{kl} f)(0) \le -d_2 <0, 
\qquad \forall\, \xi \in \mathbb S^{n-1},
\end{align*}
where $d_2>0$ is an absolute constant.
\end{enumerate}
\end{thm}
\begin{proof}
This follows from Lemma \ref{Sep8_e0},  \ref{Sep8_e1}, \ref{Sep8_e2} and \ref{Sep8_e2aa}.
\end{proof}

\subsection{The regime $\|u \|_2 \sim 1$}

\begin{lem}[The regime $\| u\|_2 \sim 1$ with $\epsilon_0 \le |\hat u \cdot e_1| \le 1-
\epsilon_0$ is OK] \label{Sep19_e0}
Let $0<\epsilon_0 \ll 1$ be given. Assume $0<c_1 <c_2<\infty$ are two given constants.
Then for $m\gtrsim n$, the following hold with high probability:
The loss function $f=f(u)$ has no critical points in the regime:
\begin{align*}
\Bigl\{ u= \sqrt R \hat u: \;  c_1<R <c_2, \; \epsilon_0 \le |\hat u \cdot e_1|
\le 1 -\epsilon_0  \Bigr\}.
\end{align*}
More precisely, introduce the parametrization $\hat u = e_1 \cos \theta + e^{\perp} \sin \theta$,
where $\theta \in [0,\pi]$ and $e^{\perp} \in \mathbb S^{n-1}$ satisfies $e^{\perp}
\cdot e_1=0$. Then  in the aforementioned regime, we have
\begin{align*}
|\partial_{\theta} f | \ge \alpha_1>0,
\end{align*}
where $\alpha_1$ depends only on ($\beta$, $\epsilon_0$, $c_1$, $c_2$).
\end{lem}
\begin{proof}
See appendix.
\end{proof}

\begin{lem}[The regime $\| u\|_2 \sim 1$ with $ |\hat u \cdot e_1| \le 
\epsilon_1$ is OK] \label{Sep19_e1}
Let $0<\epsilon_1 \ll 1$ be a sufficiently small constant.
 Assume $0<c_1 <c_2<\infty$ are two given constants.
Then for $m\gtrsim n$, the following hold with high probability:
Consider the regime
\begin{align*}
\Bigl\{ u= \sqrt R \hat u: \;  c_1<R <c_2, \; |\hat u \cdot e_1|
\le \epsilon_1  \Bigr\}.
\end{align*}
Introduce the parametrization $\hat u = e_1 \cos \theta + e^{\perp} \sin \theta$,
where $\theta\in[0,\pi]$ and $e^{\perp} \in \mathbb S^{n-1}$ satisfies $e^{\perp}
\cdot e_1=0$. Then  in the aforementioned regime, we have
\begin{align*}
\partial_{\theta \theta } f  \le  -\alpha_2<0,
\end{align*}
where $\alpha_2>0$ depends only on ($\beta$, $\epsilon_1$, $c_1$, $c_2$).
\end{lem}
\begin{proof}
See appendix.
\end{proof}

\begin{thm}[The regime $\|u\|_2 \sim 1 $, $||\hat u\cdot e_1|-1|\le \epsilon_0$,
 $|\|u\|_2-1|\ge c(\epsilon_0)$ is OK] \label{Sep19_e3}
Let $0<R_1<1<R_2<\infty$ be given constants. 
Let $0<\epsilon_0\ll 1$ be a given sufficiently small constant and consider the regime $ \Bigl | |\hat u \cdot e_1|-1
\Bigr| \le \epsilon_0$ with $R_1 \le \| u\|_2^2 \le R_2$.  There exists 
a constant $c_0=c_0(\epsilon_0,R_1,R_2, \beta)>0$ 
which tends to zero as $\epsilon_0\to 0$ such that the following
hold:
For $m\gtrsim n$,  with high probability it holds that (below $u=\sqrt R \hat u$)
\begin{align*}
&\partial_{R} f <0,  \quad\forall\, R_2\le R \le 1-c_0 , \;\;\text{}\forall\, \hat u \in \mathbb S^{n-1} \text{\quad with\quad  $| |\hat u \cdot e_1|-1| \le \epsilon_0$};\\
&\partial_{R} f >0,  \quad\forall\, 1+ c_0 \le  R \le R_1,\; \;\text{}\forall\, \hat u \in \mathbb S^{n-1} \text{\quad with \quad $| |\hat u \cdot e_1| -1| \le \epsilon_0$}.
\end{align*}
\end{thm}
\begin{proof}
We first consider the regime $R\ge 1+c$.  Let $\phi \in C_c^{\infty}(\mathbb R)$ be an even
function satisfying
$0\le \phi(x) \le 1$ for all $x$, $\phi(x)=1$ for $|x|\le 1$ and $\phi(x)=0$
for $|x| >2$.  By using \eqref{Sep8e10a}, we have
\begin{align}
&\partial_{R} f  \notag \\
\ge & 
\frac 1m \sum_{k=1}^m
\frac 
{ R^2  \beta (a_k\cdot \hat u)^4  \phi(\frac {a_k\cdot \hat u} K)
+2 R (a_k\cdot \hat u)^4 \phi(\frac {a_k\cdot \hat u} K) (a_k\cdot e_1)^2
-\beta (a_k\cdot e_1)^4 -2(a_k\cdot e_1)^4 (a_k\cdot \hat u)^2}
{
(\beta R+ (a_k\cdot e_1)^2 )^2}. \label{Sep19_e3_00a}
\end{align}
By taking $K$ sufficiently large, we can easily obtain
\begin{align*}
\mathbb E  (1- \phi( \frac {a \cdot e_1} K) ) (1+(a\cdot e_1)^2) \ll 1,
\end{align*}
where $a\sim \mathcal N(0, \operatorname{I_n})$.  For fixed $K$, it is not difficult
to check that the lower bound \eqref{Sep19_e3_00a} are OK for union
bounds and they can be made close to the expectation with high probability, 
uniformly in $R \sim 1$ 
and $\hat u \in \mathbb S^{n-1}$.  The perturbation argument (i.e.
estimating the error terms coming from replacing $a_k\cdot \hat u$
by $a_k\cdot e_1$ and so on) becomes rather
easy after taking the expectation. It is then not difficult to show that 
\begin{align*}
\partial_R f >0,
\end{align*}
for $R \ge 1+c(\epsilon_0)$. 

Next we turn to the regime $R_2 \le R \le 1-c(\epsilon_0)$. Without loss of
generality we may assume $|1-\hat u\cdot e_1| \le \epsilon_0$.
The idea is to exploit the
decomposition used in the proof of Lemma \ref{Sep8_e2aa}. Namely
using  $\partial_R f =0$ and $\partial_{u_1} f =0$, we have
\begin{align*}
R \underbrace{\frac 1 m  \sum_{k=1}^m\frac { R \beta Z_k^4
+2 Z_k^4 X_k^2} 
{ (\beta R+ X_k^2)^2}}_{=: A_1}
&= \underbrace{\frac 1m \sum_{k=1}^m
\frac { \beta X_k^4 + 2 X_k^4  Z_k^2}
{(\beta R +X_k^2 )^2}}_{=:B_1};  \\
-R \underbrace{ \widehat u_1 
\frac 1m \sum_{k=1}^m
\frac {\beta R Z_k^4} {(\beta R+X_k^2)^2}}_{=:A_2}
+R \cdot \underbrace{2 
\frac 1m \sum_{k=1}^m
\frac {Z_k^3 X_k} {\beta R+ X_k^2}}_{=:A_3}
&=\underbrace{\beta \widehat u_1 \frac 1m \sum_{k=1}^m
\frac {-2R Z_k^2 X_k^2+X_k^4}
{(\beta R+X_k^2)^2}
+2
\frac 1m \sum_{k=1}^m 
\frac {Z_k X_k^3} {\beta R+ X_k^2}}_{=:B_2}.
\end{align*}
It is not difficult to check that with high probability, we have $B_1 \sim 1$, $B_2 \sim 1$,
and
\begin{align*}
\Bigl |\frac {B_2} {B_1} -1 \Bigr| \le \eta(\epsilon_0) \ll 1,
\qquad\forall\, R_2\le R\le 1, \; \forall\, \hat u \in \mathbb S^{n-1}
\text{ with $|\hat u \cdot e_1 -1| \le \epsilon_0$},
\end{align*}
where $\eta(\epsilon_0) \to 0$ as $\epsilon_0\to 0$.  We then obtain
\begin{align*}
A_1 = \Bigl( 1+ O(\eta(\epsilon_0)) \Bigr) (-A_2+A_3).
\end{align*}
From this it is easy (similar to an argument used in the proof of Lemma \ref{Sep8_e2aa})
to derive that 
\begin{align*}
A_1 +A_2 +|A_3| \lesssim 1.
\end{align*}
Now note that the pre-factor of $A_2$ is $\hat u_1=1+O(\epsilon_0)$. By using the relation
\begin{align*}
A_1 +A_2 -A_3 = O(\eta(\epsilon_0) ),
\end{align*}
 we obtain
\begin{align*}
\frac 1m \sum_{k=1}^m
\frac {Z_k^4} {\beta R+ X_k^2}
-\frac 1m \sum_{k=1}^m
\frac {Z_k^3 X_k} {\beta R+ X_k^2} = O(\eta(\epsilon_0) ).
\end{align*}
By using localization (i.e. decomposing $Z_k^3 X_k=Z_k^3 \phi(\frac {Z_k} M) X_k
+ Z_k^3 (1- \phi(\frac {Z_k} M ) ) X_k$), H\"older and taking $M$ sufficiently large, one can
then derive that (with high probability)
\begin{align*}
\Bigl |\frac 1m \sum_{k=1}^m \frac {Z_k^4 -X_k^4} { \beta R+ X_k^2} \Bigr|
+ \Bigl|
\frac 1m \sum_{k=1}^m \frac {Z_k^3 X_k -X_k^4} { \beta R+ X_k^2} \Bigr| = O(\eta_1(\epsilon_0)), \qquad\forall\, R_2\le R\le 1, \; \forall\, \hat u \in \mathbb S^{n-1}
\text{ with $|\hat u \cdot e_1 -1| \le \epsilon_0$},
\end{align*}
where $\eta_1(\epsilon_0)\to 0$ as $\epsilon_0\to 0$.  It then follows easily that
(with high probability)
\begin{align*}
\Bigl| \frac 1m \sum_{k=1}^m \frac {(Z_k-X_k)^4} {\beta R+X_k^2} 
\Bigr| = O(\eta_2(\epsilon_0) ), 
\qquad\forall\, R_2\le R\le 1, \; \forall\, \hat u \in \mathbb S^{n-1}
\text{ with $|\hat u \cdot e_1 -1| \le \epsilon_0$},
\end{align*}
where $\eta_2(\epsilon_0) \to 0$ as $\epsilon_0 \to 0$. 

Now observe that for $A_1$, we have
\begin{align*}
|Z_k^4 -X_k^4| & \le |Z_k-X_k| ( O(|Z_k|^3) +O(|X_k|^3) ) \notag \\
& \le  C_{\epsilon} |Z_k-X_k|^4 + \epsilon\cdot ( O(|Z_k|^4 )+O( X_k^4) ),
\end{align*}
where $C_{\epsilon}>0$ depends only on $\epsilon$. Clearly by taking $\epsilon>0$ 
sufficiently small and using the derived quantitative estimates preceding this
paragraph, we can guarantee that (with high probability)
\begin{align*}
\Bigl | A_1 - B_1 \Bigr| \ll 1, \qquad\forall\, R_2\le R\le 1, \; \forall\, \hat u \in \mathbb S^{n-1}
\text{ with $|\hat u \cdot e_1 -1| \le \epsilon_0$}.
\end{align*}
It follows that we must have $|R-1| \ll 1$ for a potential critical point.  By using
\eqref{Sep8e10a} we have $\partial_R f (R=0) <0$. By using \eqref{Sep8e10b}
we have $\partial_{RR} f >0$.  Since we have shown $\partial_R f >0$ for $R >1+c(\epsilon_0)$,
it then follows that  $\partial_R f=0$ occurs at a unique point $|R-1|\ll 1$ and
$\partial_R f <0$ for $ R<1-c(\epsilon_0)$ provided $c(\epsilon_0)$ is
suitably re-defined. 
\end{proof}

We now show restrictive convexity of the loss function $f(u)$ near the global
minimizer $u=\pm e_1$.

\begin{thm}[Restrictive convexity near the global minimizer] \label{Sep19e8}
There exists $0<\epsilon_0\ll 1$  sufficiently small such that if
$m\gtrsim n$, then the following hold with high probability:

\begin{enumerate}
\item If $\|u -e_1\|_2\le \epsilon_0$ and $u\ne e_1$, then for $\xi = 
\frac {u-e_1} {\| u-e_1\|_2} \in \mathbb S^{n-1}$, we have
\begin{align*}
\sum_{i,j=1}^n \xi_i \xi_j (\partial_{ij} f)(u) \ge \gamma>0, 
\end{align*}
where $\gamma$ is a constant depending only on $\beta$. 

\item If $\|u +e_1\|_2\le \epsilon_0$, then then for $\xi = 
\frac {u+e_1} {\| u+ e_1\|_2} \in \mathbb S^{n-1}$, we have
\begin{align*}
\sum_{i,j=1}^n \xi_i \xi_j (\partial_{ij} f)(u) \ge \gamma>0, 
\end{align*}
where $\gamma$ is a constant depending only on $\beta$. 

\item Alternatively we can use the parametrization $u= \pm e_1+ t \xi$, where $\xi\in \mathbb S^{n-1}$, and $|t|\le \epsilon_0$. Then with this special parametrization, we have 
\begin{align*}
\sum_{i,j=1}^n \xi_i \xi_j (\partial_{ij} f)(u) \ge \gamma>0.
\end{align*}
Note that this includes the global minimizers $u=\pm e_1$. 

\end{enumerate}

In yet other words, $f(u)$ is restrictively convex in a sufficiently small neighborhood of $\pm e_1$.
\end{thm}
\begin{proof}
See appendix.
\end{proof}

\begin{proof}[Proof of Theorem \ref{thmEb}]
We proceed in several steps as follows.
\begin{enumerate}
\item For the regime $\|u\|_2 \ll 1$ and $\|u\|_2\gg 1$, we use 
Theorem \ref{thmSep8_1}. The point $u=0$ is a local maximum point with
strictly negative-definite Hessian. All other possible critical points must have
negative curvature direction.
\item For the regime $\|u\|_2 \sim 1$, $| |\hat u \cdot e_1|-1| \ge \epsilon_0$,
we use Lemma \ref{Sep19_e0} and \ref{Sep19_e1}. The loss function either has
a nonzero gradient, or it is a strict saddle with a negative curvature direction.
\item For the regime $\|u\|_2 \sim 1$, $ | \hat u \cdot e_1 | -1|\le \epsilon_0$,
$| \|u\|_2 -1|\ge c(\epsilon_0)$, we apply Theorem \ref{Sep19_e3}. The loss
function has nonzero gradient in this regime.
\item Finally for the regime close to the global minimizers $\pm e_1$, we use
Theorem \ref{Sep19e8} to show restrictive convexity. This ensures that $\pm e_1$
are the only minimizers. 
\end{enumerate}
\end{proof}

\section{Quotient intensity model \RNum{3}}  \label{S:model4c}
Consider for $\beta_1>0$, $\beta_2>0$,  
\begin{align} \label{model4ce1}
f(u) &= \frac 1 m \sum_{k=1}^m 
\frac{ ( (a_k\cdot u)^2 - (a_k\cdot x)^2 )^2}
{ |u|^2 +\beta_1(a_k\cdot u)^2+ \beta_2 (a_k\cdot x)^2}.
\end{align}

\begin{thm} \label{thmEc}
Let $0<\beta_1,\beta_2 <\infty$. 
Assume 
$\{a_k\}_{k=1}^m$ are i.i.d.  standard Gaussian random vectors and $x\ne 0$. 
There exist positive constants $c$, $C$ depending only on $(\beta_1,\beta_2)$, such that if $m\ge C n  $, then
with probability at least $1- e^{-cm} $ the loss function $f=f(u)$ 
defined by \eqref{model4ce1}
has no spurious local minimizers. The only global minimizer is
$\pm  x$, and the loss function is strongly convex in a neighborhood of $\pm x$.
The point $u=0$ is a local maximum point with strictly negative-definite Hessian.
All other critical points are strict saddles, i.e., each saddle point has a neighborhood
where the function has negative directional curvature.
\end{thm}

Without loss of generality we shall assume $x=e_1$ throughout the rest of the proof.
Thus we consider
\begin{align} \label{model4ce2}
f(u) &= \frac 1 m \sum_{k=1}^m 
\frac{ ( (a_k\cdot u)^2 - (a_k\cdot e_1)^2 )^2}
{ |u|^2 +\beta_1(a_k\cdot u)^2+ \beta_2 (a_k\cdot e_1)^2}.
\end{align}

\subsection{The regimes $\|u\|_2\ll 1 $ and $\|u\|_2\gg 1$ are fine}$\;$

We first investigate the point $u=0$. It is trivial to verify that $\nabla f(0)=0$ since 
$a_k\cdot e_1 \ne 0$ for all $k$ almost surely. 
\begin{lem}[$u=0$ has strictly negative-definite Hessian] \label{Sep21_e0}
We have $u=0$ is local maximum point with strictly
negative-definite Hessian. More precisely, it holds (almost surely) that
\begin{align*}
\sum_{k,l=1}^n \xi_k \xi_l (\partial_{kl} f)(0)
\le -d_1, \quad\forall\, \xi \in \mathbb S^{n-1},
\end{align*}
where $d_1>0$ is a constant depending only on $\beta_2$.
\end{lem}
\begin{proof}
We begin by noting that  since almost surely $a_k\cdot e_1\ne 0$ for all $k$,
the function $f$ is smooth at $u=0$.  It suffices for us to consider 
(write $u=\sqrt{t} \xi$)
\begin{align*}
G(t) = \frac 1 m
\sum_{k=1}^m \frac { (t(a_k\cdot \xi)^2 - (a_k\cdot e_1)^2)^2}
{ t+ t\beta_1 (a_k\cdot \xi)^2+\beta_2(a_k\cdot e_1)^2}.
\end{align*}
By a simple computation, we have
\begin{align*}
G^{\prime}(0) =-\frac 1 {\beta_2^2} -
\frac {\beta_1+2\beta_2}{\beta_2^2} \cdot \frac 1m \sum_{k=1}^m (a_k\cdot \xi)^2.
\end{align*}
The desired conclusion then easily follows.
\end{proof}

Write  $u=\sqrt{R} \hat u$ where $\hat u \in S^{n-1}$ and $R>0$.  Denote
$X_k=a_k\cdot e_1$.  Then
\begin{align*}
f (u) = \frac 1m \sum_{k=1}^m 
\frac { \Bigl( R (a_k\cdot \hat u)^2 -X_k^2 \Bigr)^2}
{ R + \beta_1 R (a_k\cdot \hat u)^2 +\beta_2 X_k^2}.
\end{align*}
Clearly
\begin{align}
&\partial_{R} f =   
\frac 1m \sum_{k=1}^m
\frac 
{R^2((a_k\cdot \hat u)^4+\beta_1 (a_k\cdot \hat u)^6) +
2R \beta_2 (a_k\cdot \hat u)^4 X_k^2 -X_k^4-(\beta_1+2\beta_2)(a_k\cdot \hat u)^2 X_k^4}
{
( R+ R\beta_1(a_k\cdot \hat u)^2 +\beta_2 X_k^2)^2};
   \label{Sep21e1.0a}\\
&\partial_{RR} f =2\frac 1m \sum_{k=1}^m
\frac{  \Bigl( 1+ (a_k\cdot \hat u)^2 (\beta_1+\beta_2) \Bigr) X_k^4}
{ ( R+ R \beta_1(a_k\cdot \hat u)^2+\beta_2  X_k^2)^3}. \label{Sep21e1.0b}
\end{align}

\begin{lem}[The regimes $\|u\|_2\gg 1$ or $\|u\|_2\ll 1$ are OK] \label{Sep21_e1}
There exist constants $R_i=R_i(\beta_1,\beta_2)>0$, $d_i=d_i(\beta_1,\beta_2)>0$, $i=1,2$
such that the 
following hold: 
For $m\gtrsim n$,  with high probability it holds that
\begin{align*}
&\partial_{R} f \ge d_1,  \quad\forall\, R \ge R_1 , \;\text{}\forall\, \hat u \in \mathbb S^{n-1};\\
&\partial_{R} f \le  -d_2<0,  \quad\forall\, 0<R \le R_2 , \;\text{}\forall\, \hat u \in \mathbb S^{n-1}.
\end{align*}
\end{lem}
\begin{proof}
Denote $Z_k= a_k\cdot \hat u$. We first consider the regime $R\gg 1$.
 Observe that
\begin{align*}
\frac 1m \sum_{k=1}^m \frac { R^2 Z_k^4} { (R +R\beta_1 Z_k^2+\beta_2 X_k^2)^2}
& \gtrsim \frac 1 m\sum_{k=1}^m \frac {Z^4_k} { (1+Z_k^2+X_k^2)^2}
\gtrsim 1, \quad\forall\, \hat u \in \mathbb S^{n-1},
\end{align*}
where the last inequality holds for $m\gtrsim n$ with high probability. 
On the other hand we note that 
\begin{align*}
\frac 1m \sum_{k=1}^m \frac {X_k^4} { (R + R \beta_1 Z_k^2 + \beta_2 X_k^2)^2}
& \lesssim \frac 1m \sum_{k=1}^m \frac {X_k^4} { (R +X_k^2)^2} (\chi_{|X_k|\le R^{\frac 14}}
+ \chi_{|X_k|>R^{\frac 14} } )\notag \\
& \lesssim R^{-1} + \frac 1m \sum_{k=1}^m  \chi_{|X_k|>R^{\frac 14} }
 \ll 1, \qquad\forall\, \hat u \in \mathbb S^{n-1},
 \end{align*}
 where again the last inequality holds for $R$ sufficiently large, and for 
 $m\gtrsim n$ with high probability.
  Similarly we have for $R$ sufficiently large,
  \begin{align*}
  \frac 1m \sum_{k=1}^m
  \frac {(\beta_1+2\beta_2) Z_k^2 X_k^4} 
  { (R+ R\beta_1 Z_k^2+\beta_2 X_k^2)^2}
  & \lesssim R^{-1} \frac 1m\sum_{k=1}^m
  Z_k^2 + \frac 1m \sum_{k=1}^m Z_k^2 \chi_{|X_k| > R^{\frac 14} } \ll 1,
  \qquad\forall\, \hat u \in \mathbb S^{n-1}.
  \end{align*}
 Thus it follows easily that $\partial_R f \gtrsim 1$ for $R\gg 1$.
 
 Now we turn to the regime $0<R\ll 1$. 
 First we note that the main negative term is OK.
 This is
  due to the fact that for $0<R\le 1$, we have
 (for $m\gtrsim n$ and with high probability)
 \begin{align*}
 \frac 1m \sum_{k=1}^m \frac {Z_k^2 X_k^4}
 { (R+R Z_k^2+ X_k^2)^2} \ge \frac 1m \sum_{k=1}^m
 \frac {Z_k^2 X_k^4} { (1+Z_k^2+ X_k^2)^2} \gtrsim 1, \qquad\forall\,
 \hat u \in \mathbb S^{n-1}.
 \end{align*}
 On the other hand, we have (for $m\gtrsim n$ and with high probability)
 \begin{align*}
 \frac 1m \sum_{k=1}^m
 \frac {R^2 (Z_k^4 +Z_k^6)} { (R +R Z_k^2 + X_k^2)^2}
 & \le \frac 1m \sum_{k=1}^m \frac {RZ^4_k} {R+RZ_k^2 +X_k^2} 
 \cdot (\chi_{|X_k|\ge R^{\frac 14} |Z_k|} + 
 \chi_{|X_k|<R^{\frac 14} |Z_k|} )\notag \\
 & \le R^{\frac 12} \frac 1m \sum_{k=1}^m 
  Z_k^2
 + \frac 1m \sum_{k=1}^m Z_k^2 \chi_{|X_k| <R^{\frac 14} |Z_k| }
 \notag \\
 &\le R^{\frac 12} \frac 1m \sum_{k=1}^m 
  Z_k^2
 + \frac 1m \sum_{k=1}^m Z_k^2 \chi_{|Z_k| \ge K }
 +\frac 1m \sum_{k=1}^m K^2
 \chi_{|X_k|< K R^{\frac 14} } \notag\\
 & 
  \ll 1,
 \qquad\forall\, \hat u \in \mathbb S^{n-1},
 \end{align*}
 if we first take $K$ sufficiently large 
 followed by taking $R$ sufficiently small.
 The estimate of the other term
 $\frac {R Z_k^4 X_k^2}{
 (R+R Z_k^2+ X_k^2)^2}$ is similar and we omit
 further details.
 
 Collecting the estimates, it is then clear that we can
 obtain the desired estimate for $\partial_R f$
 when $0<R\ll 1$.
\end{proof}

\subsection{The regime $\|u \|_2 \sim 1$}

\begin{lem}[The regime $\| u\|_2 \sim 1$ with $\epsilon_0 \le |\hat u \cdot e_1| \le 1-
\epsilon_0$ is OK] \label{Sep22_e0}
Let $0<\epsilon_0 \ll 1$ be given. Assume $0<c_1 <c_2<\infty$ are two given constants.
Then for $m\gtrsim n$, the following hold with high probability:
The loss function $f=f(u)$ has no critical points in the regime:
\begin{align*}
\Bigl\{ u= \sqrt R \hat u: \;  c_1<R <c_2, \; \epsilon_0 \le |\hat u \cdot e_1|
\le 1 -\epsilon_0  \Bigr\}.
\end{align*}
More precisely, introduce the parametrization $\hat u = e_1 \cos \theta + e^{\perp} \sin \theta$,
where $\theta \in [0,\pi]$ and $e^{\perp} \in \mathbb S^{n-1}$ satisfies $e^{\perp}
\cdot e_1=0$. Then  in the aforementioned regime, we have
\begin{align*}
|\partial_{\theta} f | \ge \alpha_1>0,
\end{align*}
where $\alpha_1$ depends only on ($\beta$, $\epsilon_0$, $c_1$, $c_2$).
\end{lem}
\begin{proof}
We first recall 
\begin{align*}
f (u) = \frac 1m \sum_{k=1}^m 
\frac { \Bigl( R (a_k\cdot \hat u)^2 -X_k^2 \Bigr)^2}
{ R + \beta_1 R (a_k\cdot \hat u)^2 +\beta_2 X_k^2}.
\end{align*}
Clearly $a_k\cdot \hat u = X_k \cos \theta +(a_k\cdot e^{\perp}) \sin \theta$, and 
\begin{align*}
& \partial_{\theta} (a_k\cdot \hat u) = X_k  (-\sin \theta)
+(a_k\cdot e^{\perp}) \cos \theta; \\
& \partial_{\theta\theta} ( a_k\cdot \hat u)=
- (a_k\cdot \hat u).
\end{align*}
In particular, if $\theta$ is away from the end-points $0, \pi$, then 
\begin{align*}
\partial_{\theta} (a_k\cdot \hat u) = (a_k\cdot \hat u) \cot \theta - X_k \csc \theta.
\end{align*}
We then obtain (below $Z_k=a_k\cdot \hat u$)
\begin{align*}
\partial_{\theta} f 
& =-\csc \theta \frac 1 m\sum_{k=1}^m 
\frac { 2R Z_k (-X_k^2+R Z_k^2)\cdot \Bigl(
(\beta_1+2\beta_2)X_k^2+R (2+\beta_1 Z_k^2) \Bigr)
X_k} 
{ (R+ \beta_1 R Z_k^2 +\beta_2 X_k^2)^2}  \notag \\
& \quad + \cot \theta \frac 1 m\sum_{k=1}^m 
\frac { 2R Z_k (-X_k^2+R Z_k^2)\cdot \Bigl(
(\beta_1+2\beta_2)X_k^2+R (2+\beta_1 Z_k^2) \Bigr)
Z_k} 
{ (R+ \beta_1 R Z_k^2 +\beta_2 X_k^2)^2}.
\end{align*}
Thanks to the strong damping, it is not difficult to check that for any
$\epsilon>0$, if $m\gtrsim n$, then with high probability we have
\begin{align*}
|\partial_{\theta} f - \mathbb E \partial_{\theta} f |
\le \epsilon, \qquad\forall\, c_1\le R \le c_2, \,\forall\,
\hat u \in \mathbb S^{n-1}.
\end{align*}
The desired result then follows from Lemma \ref{lemSeI:1}. 
\end{proof}

\begin{lem}[The regime $\| u\|_2 \sim 1$ with $ |\hat u \cdot e_1| \le 
\epsilon_0$ is OK] \label{Sep22_e1}
Let $0<\epsilon_1 \ll 1$ be a sufficiently small constant.
 Assume $0<c_1 <c_2<\infty$ are two given constants.
Then for $m\gtrsim n$, the following hold with high probability:
Consider the regime
\begin{align*}
\Bigl\{ u= \sqrt R \hat u: \;  c_1<R <c_2, \; |\hat u \cdot e_1|
\le \epsilon_1  \Bigr\}.
\end{align*}
Introduce the parametrization $\hat u = e_1 \cos \theta + e^{\perp} \sin \theta$,
where $\theta\in[0,\pi]$ and $e^{\perp} \in \mathbb S^{n-1}$ satisfies $e^{\perp}
\cdot e_1=0$. Then  in the aforementioned regime, we have
\begin{align*}
\partial_{\theta \theta } f  \le  -\alpha_2<0,
\end{align*}
where $\alpha_2>0$ depends only on ($\beta$, $\epsilon_1$, $c_1$, $c_2$).
\end{lem}
\begin{proof}
This is similar to the argument in the proof of Lemma \ref{Sep22_e0}. By a tedious
computation, we have
\begin{align*}
\partial_{\theta\theta} f &= \frac 1m \sum_{k=1}^m 
\frac{2R G_k}{\left(R+\beta_2 x^2+\beta_1 R Z_k^2\right)^3},
\end{align*}
where
\begin{align*}
G_k&=-8 \beta_1 R Z_k^2 \left(-X_k^2+R Z_k^2\right) \left(R+\beta_2 X_k^2+\beta_1 R Z_k^2\right) (X_k-Z_k \cos\theta)^2 \csc^2 \theta \notag \\
&-2 \left(R+\beta_2 X_k^2+\beta_1 R Z_k^2\right)^2 \Bigl(X_k^4-3 R X_k^2 Z_k^2-R Z_k^4-
2 X_k Z_k \left(X_k^2-3 R Z_k^2\right) \cos\theta \notag \\
&\quad \qquad +Z_k^2 (X_k^2-2 R Z_k^2)
 \cos2\theta \Bigr) \csc^2 \theta
\notag \\
&+\beta_1 \left(X_k^2-R Z_k^2\right)^2 \Bigl(Z_k^2 \left(R+\beta_2 X_k^2+\beta_1R Z_k^2\right)+4 \beta_1 R Z_k^2 (X_k-Z_k \cos\theta)^2 \csc^2\theta  \notag \\
&\quad -\left(R+\beta_2 X_k^2+\beta_1R Z_k^2\right) (X_k-Z_k \cos\theta)^2 
\csc^2\theta\Bigr).
\end{align*}
It is then tedious but not difficult to check that
that for any
$\epsilon>0$, if $m\gtrsim n$, then with high probability we have
\begin{align*}
|\partial_{\theta\theta} f - \mathbb E \partial_{\theta\theta} f |
\le \epsilon, \qquad\forall\, c_1\le R \le c_2, \,\forall\,
\hat u \in \mathbb S^{n-1}.
\end{align*}
The desired result then follows from Lemma \ref{lemSeI:1}. 
\end{proof}

\begin{thm}[The regime $\|u\|_2 \sim 1 $, $||\hat u\cdot e_1|-1|\le \epsilon_0$,
 $|\|u\|_2-1|\ge c(\epsilon_0)$ is OK] \label{Sep22_e3}
Let $0<c_1<1<c_2<\infty$ be given constants. 
Let $0<\epsilon_0\ll 1$ be a given sufficiently small constant and consider the regime $ \Bigl | |\hat u \cdot e_1|-1
\Bigr| \le \epsilon_0$ with $c_1 \le \| u\|_2^2 \le c_2$.  There exists 
a constant $c_0=c_0(\epsilon_0,c_1,c_2, \beta)>0$ 
which tends to zero as $\epsilon_0\to 0$ such that the following
hold:
For $m\gtrsim n$,  with high probability it holds that (below $u=\sqrt R \hat u$)
\begin{align*}
&\partial_{R} f <0,  \quad\forall\, c_2\le R \le 1-c_0 , \;\;\text{}\forall\, \hat u \in \mathbb S^{n-1} \text{with $| |\hat u \cdot e_1|-1| \le \epsilon_0$};\\
&\partial_{R} f >0,  \quad\forall\, 1+ c_0 \le  R \le c_1,\; \;\text{}\forall\, \hat u \in \mathbb S^{n-1} \text{with $| |\hat u \cdot e_1| -1| \le \epsilon_0$}.
\end{align*}
\end{thm}
\begin{proof}
We rewrite 
\begin{align*}
f(u) = \frac 1m \sum_{k=1}^m g(R, (a_k\cdot \hat u)^2, X_k^2),
\end{align*}
where
\begin{align*}
g(R,a,b) = \frac { (Ra-b)^2} { R+\beta_1 R a+ \beta_2 b}.
\end{align*}
It is not difficult to check that for $R\sim 1$, we have
\begin{align*}
| (\partial_R g)(R,a,b)- (\partial_R g)(R,b,b)|
\le \| \partial_{Ra} g \|_{\infty} |b-a| \lesssim |b-a|, \qquad\forall\, a,b \ge 0.
\end{align*}
On the other hand, note that $(\partial_R g)(1,b,b)=0$, and for $R \sim 1$,
\begin{align*}
(\partial_{RR} g)(R, b, b) =
\frac{ 2 b^2 (1+ b(\beta_1+\beta_2) )^2}
{ (R+b (\beta_2+\beta_1 R) )^3} \sim b.
\end{align*}
Thus for $R=1+\eta$, $\eta>0$ we have 
\begin{align*}
(\partial_R g )(R, a, b) & \ge \partial_R g (R, b,b) - \gamma_1 |b-a| \notag \\
& \ge \gamma_2 \cdot \eta \cdot b - \gamma_1 |b-a|,
\end{align*}
where $\gamma_1>0$, $\gamma_2>0$ are constants depending only 
on ($\beta_1$, $\beta_2$, $c_1$, $c_2$).  The desired result (for $\partial_R f >0$ when $R\to 1+$)
then follows from this and simple application of Bernstein's inequalities. The estimate for the regime $R\to 1-$ is similar. We omit
the details.
\end{proof}
\begin{thm}[Strong convexity near the global minimizer] \label{Sep24e1}
There exist $0<\epsilon_0\ll 1$ and a positive constant $\gamma$ such that if
$m\gtrsim n$, then the following hold with high probability:
\begin{enumerate}
\item If $\|u -e_1\|_2\le \epsilon_0$, then 
\begin{align*}
\sum_{i,j=1}^n \xi_i \xi_j (\partial_{ij} f)(u) \ge \gamma>0, \qquad \forall\, \xi \in \mathbb
S^{n-1}.
\end{align*}
\item If $\|u +e_1\|_2\le \epsilon_0$, then 
\begin{align*}
\sum_{i,j=1}^n \xi_i \xi_j (\partial_{ij} f)(u) \ge \gamma>0, \qquad \forall\, \xi \in \mathbb
S^{n-1}.
\end{align*}
\end{enumerate}
In yet other words, $f(u)$ is strongly convex in a sufficiently small neighborhood of $\pm e_1$.
\end{thm}
\begin{proof}
See appendix. 
\end{proof}

Finally we complete the proof of Theorem \ref{thmEc}.

\begin{proof}[Proof of Theorem \ref{thmEc}]
We proceed in several steps. All the statements below hold under the assumption
that $m\gtrsim n$ and with high probability. 

\begin{enumerate}
\item For $u=0$, we use Lemma \ref{Sep21_e0}.  In particular $u=0$ is a local maximum
point with strictly negative Hessian.

\item For $\| u\|_2\ll 1$ or $\| u\|_2 \gg 1$, we use Lemma \ref{Sep21_e1}. The loss
functions has a nonzero gradient ($\partial_R f \ne 0$) in this regime.

\item For $\|u\|_2 \sim 1$ with $\epsilon_0 \le |\hat u\cdot e_1| \le 1-\epsilon_0$,
we use Lemma \ref{Sep22_e0} to show that the loss function has a nonzero gradient
($\partial_{\theta} f\ne 0$)
 in this regime.

\item For $\|u\|_2 \sim 1$ with $|\hat u \cdot e_1|\le \epsilon_0$, by Lemma
\ref{Sep22_e1}, the loss function has a negative curvature direction
(i.e. $\partial_{\theta\theta} f<0$) in this regime.

\item For $\|u\|_2 \sim 1$, $||\hat u\cdot e_1|-1| \le \epsilon_0$,
$|  \|u\|_2-1 |\ge c(\epsilon_0)$,  Theorem \ref{Sep22_e3} shows that
the gradient of the loss function does not vanish (i.e. $\partial_R f \ne 0$).

\item For $\| u\pm e_1 \| \ll 1$, Theorem \ref{Sep24e1} gives the strong convexity
in the full neighborhood. 
\end{enumerate}
It is not difficult to check that the above 6 scenarios cover the whole of $\mathbb R^n$.
We omit further details.
\end{proof}

\section{Numerical Experiments} \label{S:numerics}

In this section, we demonstrate the numerical efficiency of our estimators by simple gradient descent and compare their performance with other competitive algorithms. 
Our Quotient intensity models are:\\

QIM1:
\begin{align*}
\min_{u\in \R^n} \qquad f(u) &= \frac 1 m \sum_{k=1}^m 
\frac{ ( (a_k\cdot u)^2 - (a_k\cdot x)^2 )^2}
{(a_k\cdot x)^2}.
\end{align*}

QIM2:
\begin{align*}
\min_{u\in \R^n} \qquad f(u) &= \frac 1 m \sum_{k=1}^m 
\frac{ ( (a_k\cdot u)^2 - (a_k\cdot x)^2 )^2}
{\beta |u|^2 +(a_k\cdot x)^2}.
\end{align*}

QIM3:
\begin{align*} 
\min_{u\in \R^n} \qquad f(u) &= \frac 1 m \sum_{k=1}^m 
\frac{ ( (a_k\cdot u)^2 - (a_k\cdot x)^2 )^2}
{ |u|^2 +\beta_1(a_k\cdot u)^2+ \beta_2 (a_k\cdot x)^2}.
\end{align*}

We have shown theoretically that any gradient descent algorithm will not get trapped in a local minimum for the estimators above. Here we present numerical experiments to show that the estimators perform very well with randomized initial guess.
   
We test the  performance of our  QIM2 and QIM3 and compare with SAF \cite{2020a}, Trust Region \cite{sun2016complete},  WF \cite{WF}, TWF \cite{TWF} and TAF \cite{TAF}. Here, it is worth emphasizing that random initialization is used for SAF, Trust Region \cite{sun2016complete} and our  QIM2, QIM3 algorithms while all other algorithms have adopted a spectral initialization. 

\subsection{Recovery of 1D Signals}
In our numerical experiments, the target vector $x\in \Rn$ is chosen randomly from the standard Gaussian distribution and the measurement vectors $ a_i, \,i=1,\ldots,m$ are generated randomly from standard Gaussian distribution or CDP model.  For the real Gaussian case, the signal $ x \sim  \mathcal{N}(0,I_n)$ and measurement vectors $ a_i \sim  \mathcal{N}(0,I_n)$ for $i=1,\ldots,m$. For the complex Gaussian case, the signal $ x \sim  \mathcal{N}(0,I_n)+i  \mathcal{N}(0,I_n)$ and measurement vectors $ a_i  \sim \mathcal{N}(0,I_n/2)+i \mathcal{N}(0,I_n/2)$. For the CDP model, we use masks of octanary patterns as in \cite{WF}.  For simplicity, our parameters and step size are fixed for all experiments. Specifically, we adopt parameter $\beta=1$ and step size $\mu=0.4$ for QIM2 and choose the parameter $\beta_1=0.1, \beta_2=1$, step size $\mu=0.3$ for QIM3.
  For Trust Region, WF, TWF and TAF, we use the codes provided in the original papers with suggested parameters.
 
 \begin{figure}[H]
\centering
    \subfigure[]{
     \includegraphics[width=0.45\textwidth]{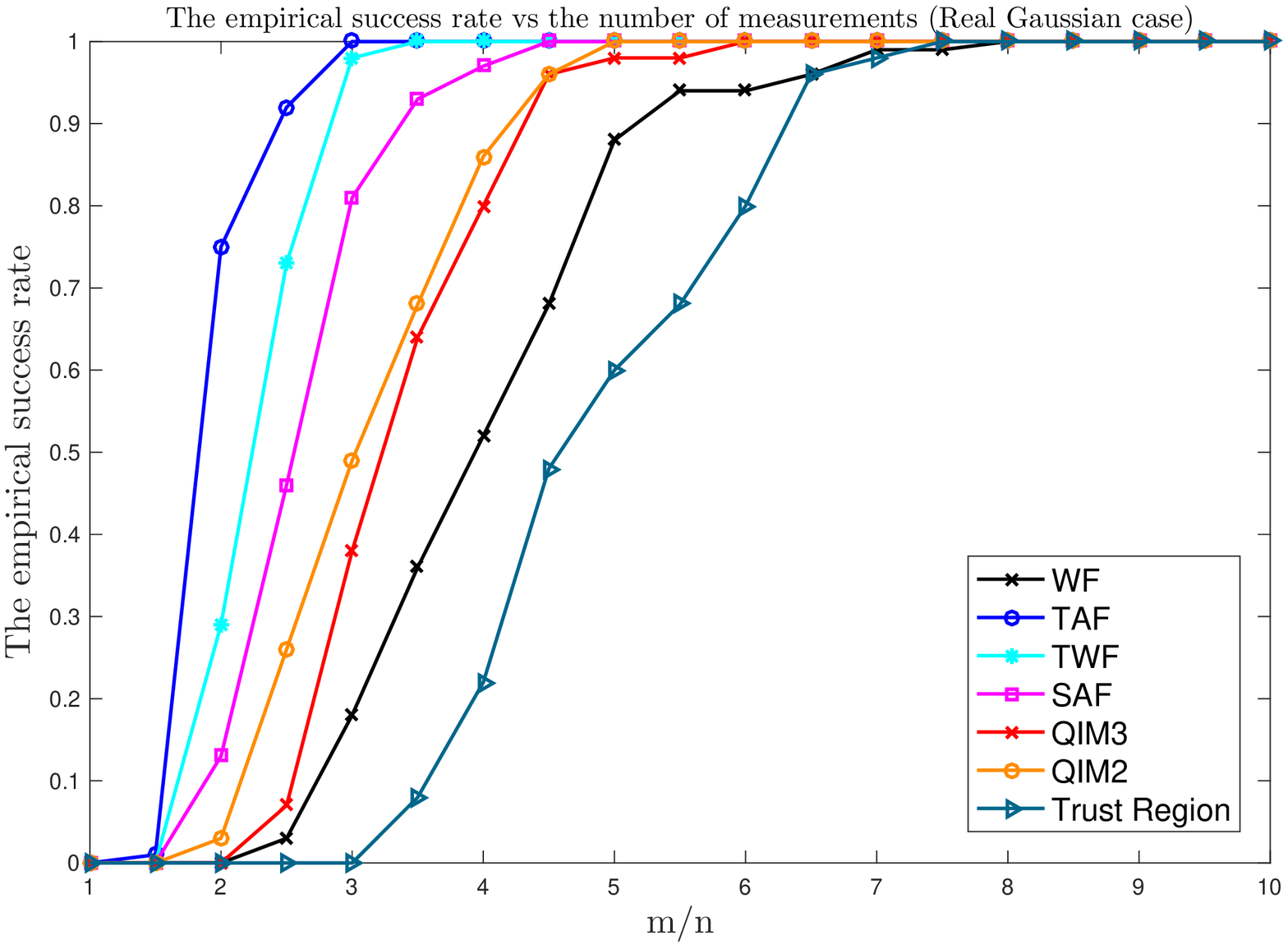}}
\subfigure[]{
     \includegraphics[width=0.45\textwidth]{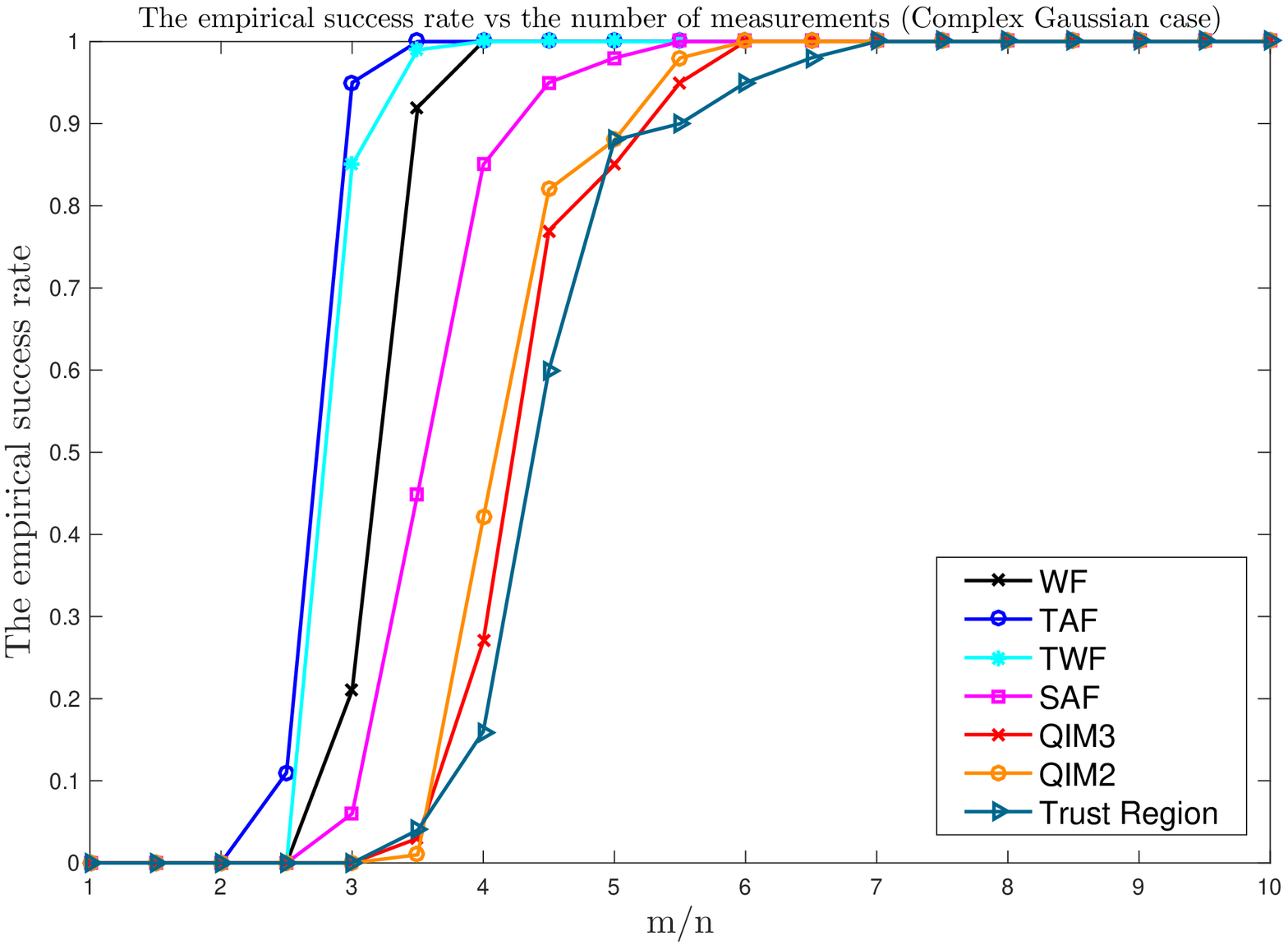}}
   \subfigure[]{
     \includegraphics[width=0.45\textwidth]{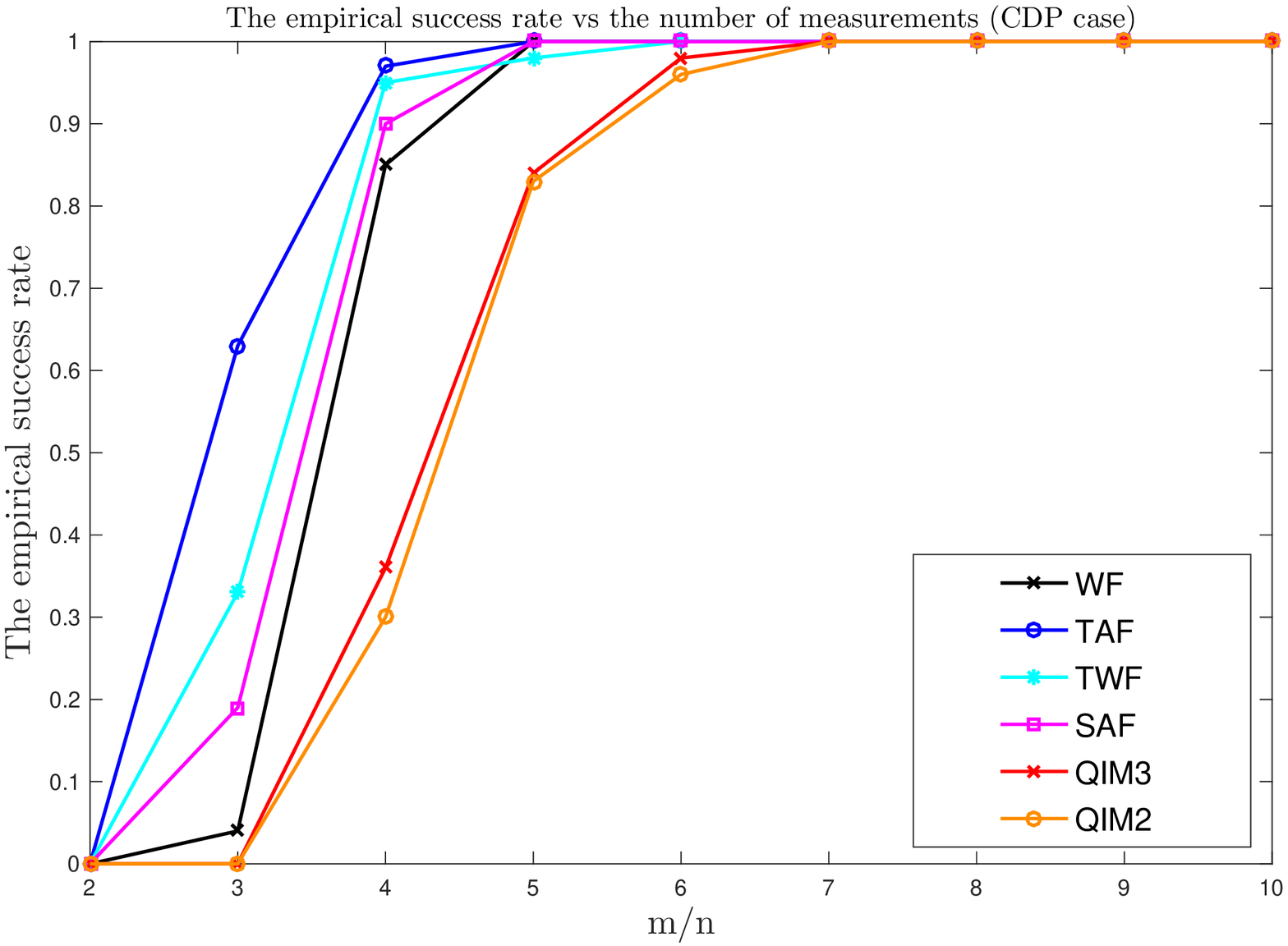}}  
\caption{ The empirical success rate for different $m/n$ based on $100$ random trails. (a) Success rate for real Gaussian case, (b) Success rate for complex Gaussian case, (c) Success rate for CDP case.}
\label{figure:succ}
\end{figure}
\begin{example}{\rm 
In this example, we test the empirical success rate of QIM2, QIM3 versus the number of measurements.  We conduct the experiments for the real Gaussian, complex Gaussian and CDP cases, respectively.
We choose $n=128$ and the maximum number of iterations is $T=2500$.   For real and complex Gaussian cases, we vary $m$ within the range $[n,10n]$. For CDP case, we set the ratio $m/n=L$ from $2$ to $10$.
For each $m$, we run $100$ times trials to calculate the success rate. Here, we say a trial to have successfully reconstructed the target signal if the relative error satisfies $\mbox{dist}(u_{T}-x)/\norm{x} \le 10^{-5}$.
The results are plotted in Figure \ref{figure:succ}. 
It can be seen that
$6n$ Gaussian phaseless measurement  or $7$ octanary patterns are enough for exactly recovery for QIM2 and QIM3.
}
\end{example}

\begin{example}
{\rm In this example, we compare the convergence rate of QIM2, QIM3 with those of SAF, WF, TWF, TAF for real Gaussian and complex Gaussian cases. We choose $n=128$ and $m=6n$. The results are presented in Figure \ref{figure:relative_error}. We can see that our algorithms perform well comparing with state-of-the-art algorithms with spectral initialization.
\begin{figure}[H]
\centering
\subfigure[]{
     \includegraphics[width=0.45\textwidth]{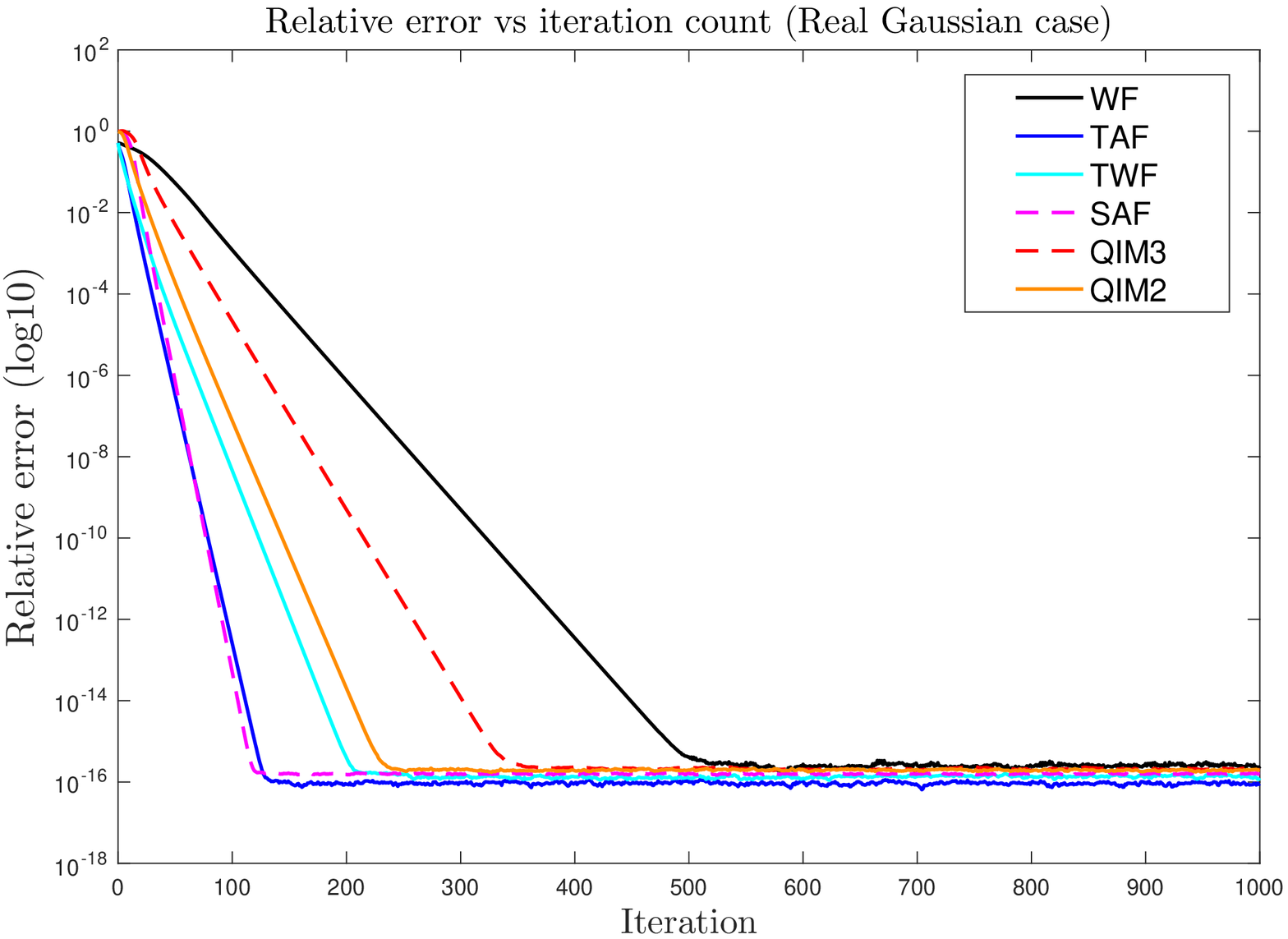}}
\subfigure[]{
     \includegraphics[width=0.45\textwidth]{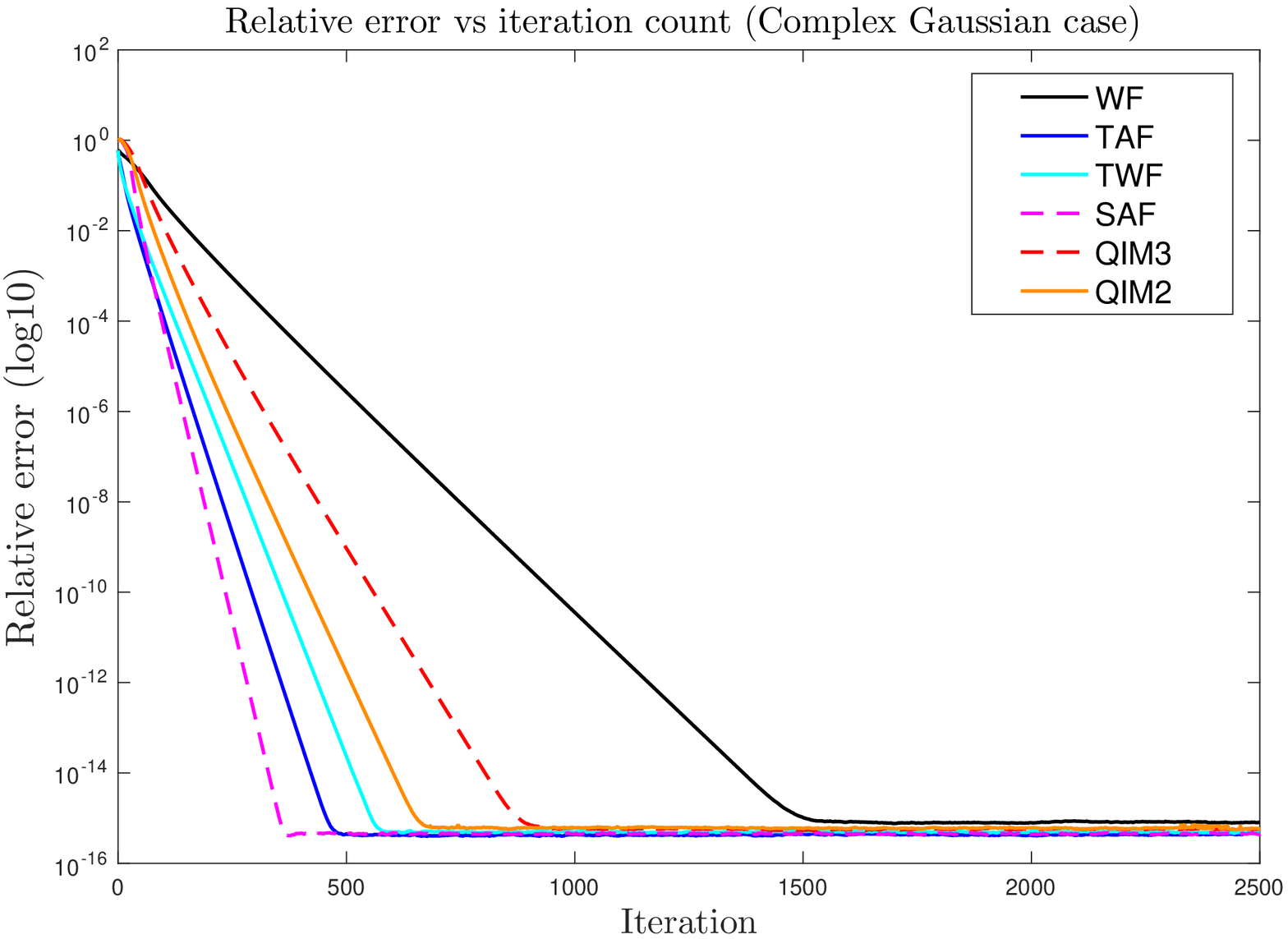}}
\caption{ Relative error versus number of iterations for QIM, SAF, WF, TWF, and TAF method: (a) Real-valued signals; (b) Complex-valued signals.}
\label{figure:relative_error}
\end{figure}
}
\end{example}

\begin{example}{\rm
In this example, we compare the time elapsed and the iteration needed for WF, TWF, TAF, SAF and our QIM2, QIM3 to achieve the relative error  $10^{-5}$ and $10^{-10}$, respectively.  We choose $n=1000$ with $m=8n$. We adopt the same spectral initialization method for WF, TWF, TAF and the initial guess is obtained by power method with $50$ iterations.  We run $50$ times trials to calculate the average time elapsed and iteration number for those algorithms. The results are shown in Table \ref{tab:performance_comparison}. The numerical results show that QIM3 takes around $27$ and $50$ iterations to escape the saddle points for the real and complex Gaussian cases, respectively. 

}
\end{example}

\begin{table}[tp]
  \centering
  \fontsize{12}{16}\selectfont
  \caption{Time Elapsed and Iteration Number among Algorithms on Gaussian Signals with $n=1000$.}
  \label{tab:performance_comparison}
    \begin{tabular}{|c|c c|cc|cc|cc|}
    \hline
    \multirow{2}{*}{Algorithm}&
    \multicolumn{4}{c|}{Real Gaussian}&\multicolumn{4}{c|}{ Complex Gaussian }\cr\cline{2-9}
    &\multicolumn{2}{c|}{$10^{-5}$ }&\multicolumn{2}{c|}{$10^{-10}$ }& \multicolumn{2}{c|}{$10^{-5}$ } &\multicolumn{2}{c|}{$10^{-10}$ }\cr \hline
    & Iter & Time(s) & Iter & Time(s) & Iter & Time(s) & Iter & Time(s) \cr \hline
   SAF & 44&\bf{0.1556} &68 &\bf{0.2276} &113&\bf{1.3092} & 190 &\bf{2.3596} \cr\hline
    QIM2 &58&2.0589&117& 3.7204 &155&21.6235& 314&37.1972\cr\hline
    QIM3 &88&2.4423&161& 4.2229 &211&30.2235& 422&48.1972\cr\hline
    WF &125&4.4214& 229 &6.3176 &304&34.6266& 655&86.6993\cr \hline
    TAF &29&0.2744&60&0.3515 &100&1.7704& 211 &2.7852\cr \hline
    TWF&40&0.3181&87&0.4274&112&1.9808& 244&3.7432\cr \hline
    Trust Region &\bf{21}&2.9832&\bf{29}&4.4683&{\bf 33}&19.1252& \bf{42}&29.0338\cr \hline
    \end{tabular}
\end{table}


\subsection{Recovery of Natural Image}
We next compare the performance of the above algorithms on recovering a natural image from masked Fourier intensity  measurements. The image is the Milky Way Galaxy with resolution $1080 \times 1920$. The colored image has RGB channels. We use $L=20$ random octanary patterns to obtain the Fourier intensity measurements for each R/G/B channel as in \cite{WF}. Table \ref{tab:performance_comp_image} lists the averaged time elapsed and the iteration needed to achieve the relative error  $10^{-5}$ and $10^{-10}$  over the three RGB channels. We can see that our algorithms have good performance comparing with state-of-the-art algorithms with spectral initialization. 
%

\begin{table}[tp]
  \centering
  \fontsize{13}{16}\selectfont
  \caption{Time Elapsed and Iteration Number among Algorithms on Recovery of Galaxy Image.}
  \label{tab:performance_comp_image}
    \begin{tabular}{|c|c c|cc|}
    \hline
    \multirow{2}{*}{Algorithm}&
    \multicolumn{4}{c|}{The Milky Way Galaxy}\cr\cline{2-5}
    &\multicolumn{2}{c|}{$10^{-5}$ }&\multicolumn{2}{c|}{$10^{-10}$ }\cr \hline
    & Iter & Time(s) & Iter & Time(s) \cr \hline
   SAF & 92 &\bf{202.47} &148 &\bf{351.21} \cr\hline
    QIM2 &168&351.32&282& 601.68 \cr\hline
    QIM3 &173&371.59&296& 709.21 \cr\hline
    WF &158 &381.7 & 277 &621.63 \cr \hline
    TAF &\bf{65} &223.89&\bf{122}&368.22 \cr \hline
    TWF&68 &315.14&145&566.84\cr \hline
    \end{tabular}
\end{table}

\subsection{ Recovery of signals with noise}
We now demonstrate the robustness of QIM2, QIM3 to noise and compare them with SAF,  WF, TWF, TAF. We  consider the noisy model $y_i=\abs{\nj{ a_i, x}}+\eta_i$ and add different level of Gaussian noises to explore the relationship between the signal-to-noise rate (SNR) of the measurements and the mean square error (MSE) of the recovered signal. Specifically, SNR and MSE are evaluated by
\[
\mbox{MSE}:= 10 \log_{10} \frac{\mbox{dist}^2(u,x)}{\norms{x}^2} \quad \mbox{and} \quad \mbox{SNR}=10 \log_{10} \frac{\sum_{i=1}^m \abs{a_i^\T x}^2}{\norms{\eta}^2},
\]
where $u $ is the output of the algorithms given above after $2500$ iterations. We choose $n=128$ and $m=8n$. The SNR varies from $20$db to $60$db. The result is shown in Figure \ref{figure:SNR}. We can see that our algorithms are stable for noisy phase retrieval.

\begin{figure}[H]
\centering
\subfigure[]{
     \includegraphics[width=0.45\textwidth]{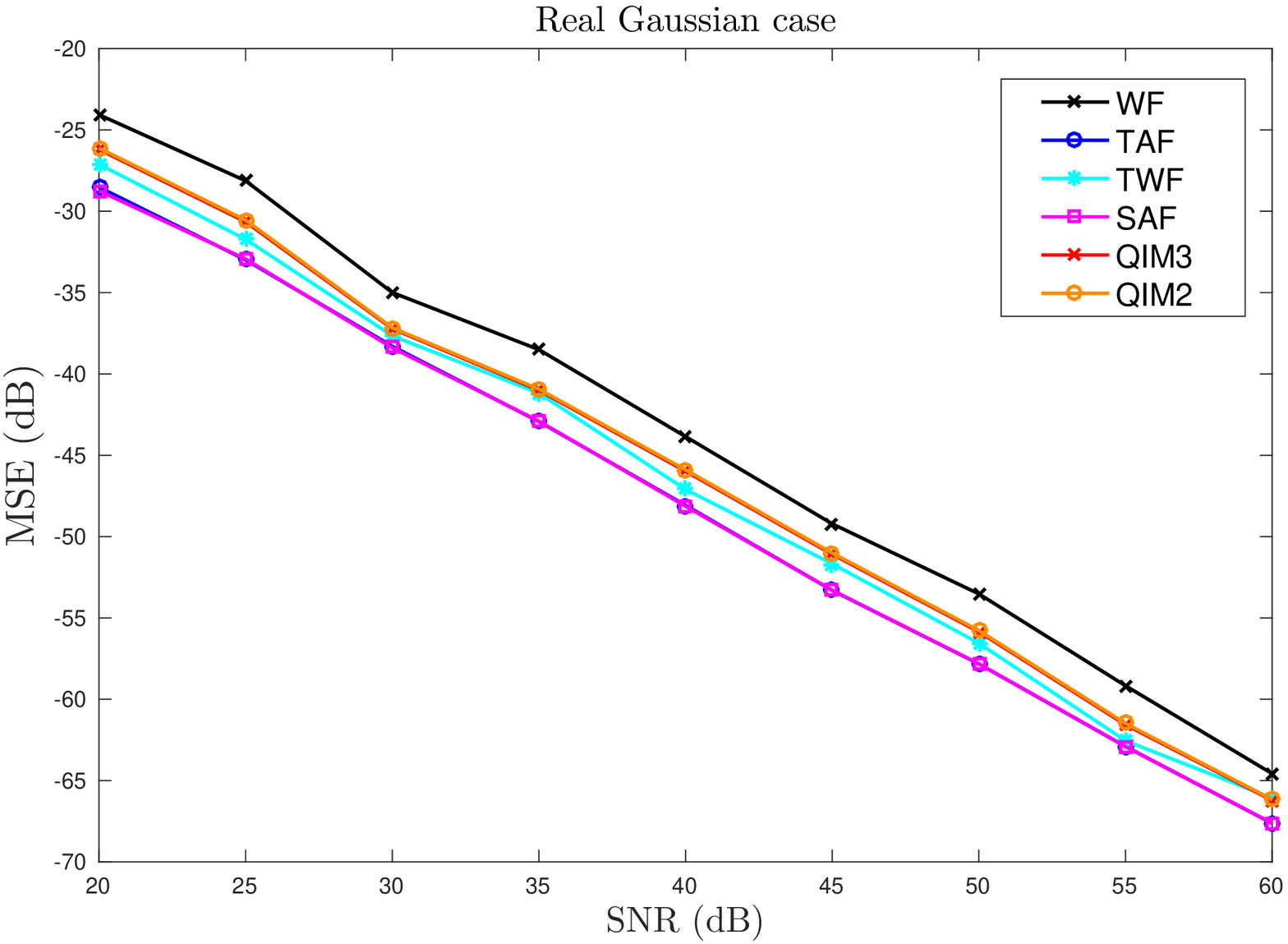}}
\subfigure[]{
     \includegraphics[width=0.45\textwidth]{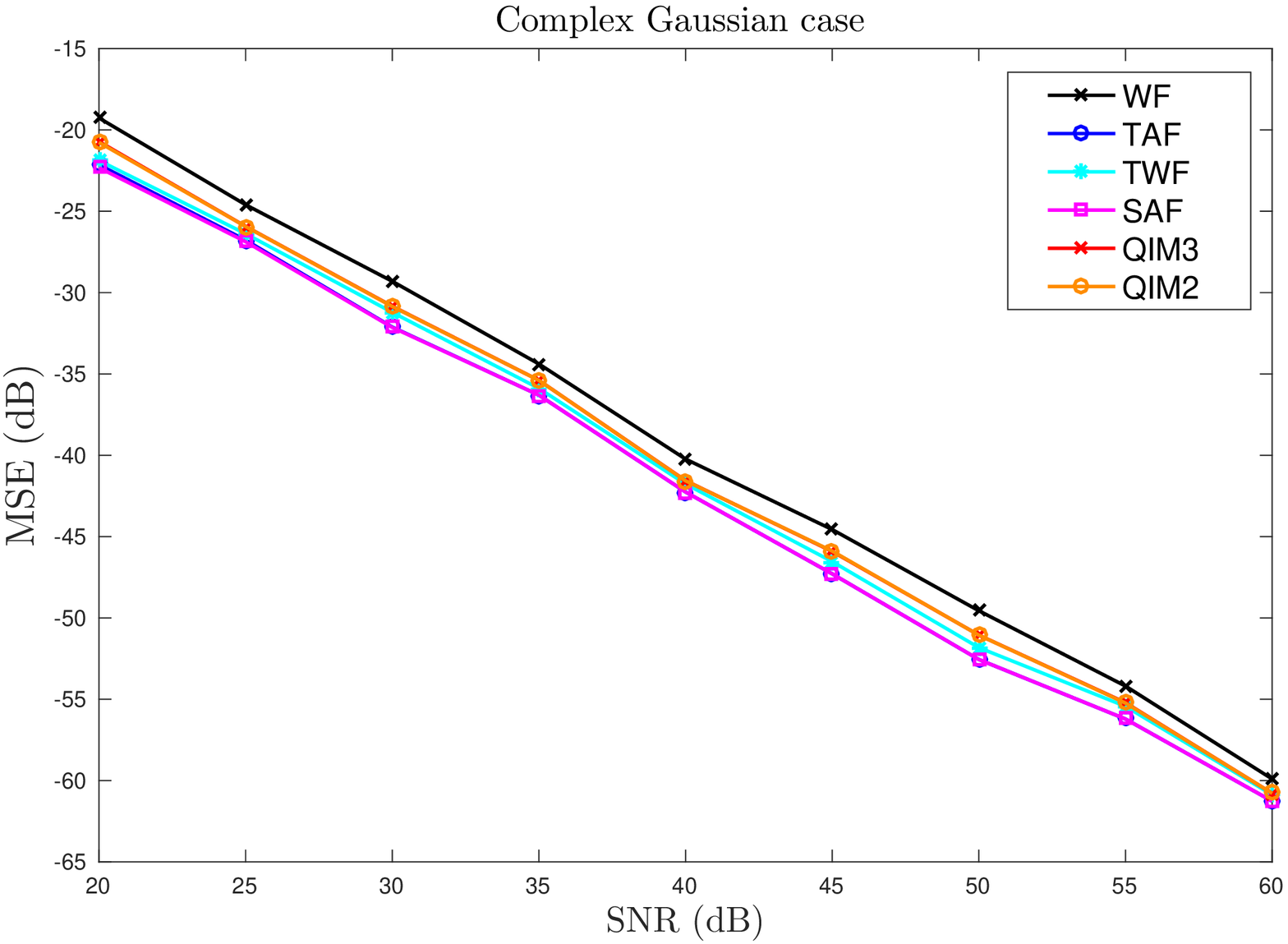}}
\caption{ SNR versus relative MSE on a dB-scale under the noisy Gaussian model: (a) Real Gaussian case; (b)  Complex Gaussian case.}
\label{figure:SNR}
\end{figure}

 \appendix
 \renewcommand{\appendixname}{Appendix~\Alph{section}}
 
\section{ Technical estimates for Section \ref{S:model4a}}
\begin{lem} \label{lemSe7_0a}
Let $\phi \in C_c^{\infty}(\mathbb R)$ satisfies
$0\le \phi(x) \le 1$ for all $x$, $\phi(x)=1$ for $|x| \le 1$
and $\phi(x)=0$ for $|x|\ge 2$. There 
exist $\epsilon>0$ sufficiently small, and
$N$ sufficiently large such that 
\begin{align*}
\mathbb E \frac {(a\cdot \xi)^2 (a\cdot e_1)^2}
{\epsilon+ (a\cdot e_1)^2}
\phi(\frac {a \cdot \xi} N)
\ge 0.99, \qquad\forall\, \xi \in \mathbb S^{n-1}, 
\end{align*}
where $a \sim \mathcal N (0, \operatorname{I_n} )$.
\end{lem}
\begin{proof}
We first show that there exist $\epsilon>0$, such that
\begin{align} \label{lemSe7_0a.1}
\mathbb E \frac {(a\cdot \xi)^2 (a\cdot e_1)^2}
{\epsilon +(a\cdot e_1)^2} \ge 0.995, \qquad\forall\, \xi \in \mathbb S^{n-1}.
\end{align}
Clearly it suffices for us to show
\begin{align} \label{lemSe7_0a.2}
\sup_{\xi \in \mathbb S^{n-1}} \mathbb E  \frac {\epsilon (a\cdot \xi)^2}
{\epsilon+(a\cdot e_1)^2} \to 0, \quad \text{as $\epsilon\to 0$}.
\end{align}
Observe that $\xi = s e_1+ \sqrt{1-s^2} e_1^{\perp}$, $|s|\le 1$, 
$e^{\perp} \cdot e_1=0$. Thus  denoting $X$ and $Y$
as two independent standard Gaussian random variables with mean zero
and unit variance, we have
\[
\sup_{\xi \in \mathbb S^{n-1}} \mathbb E  \frac {\epsilon (a\cdot \xi)^2}
{\epsilon+(a\cdot e_1)^2}  \lesssim \mathbb E \frac {\epsilon X^2 } {\epsilon +X^2}
+\mathbb E  \frac {\epsilon Y^2} {\epsilon+X^2}
\lesssim  \epsilon + \mathbb E \frac {\epsilon} {\epsilon +X^2}
\lesssim \sqrt{\epsilon},
\]
where in the last inequality we used the fact that
\begin{align*}
\int_{|x|\le 1} \frac {\epsilon}{\epsilon+x^2} dx \sim \sqrt{\epsilon}.
\end{align*}
Thus \eqref{lemSe7_0a.2} and \eqref{lemSe7_0a.1} hold.
Now $\epsilon$ is fixed. To show the final inequality, we note that
\[
 \mathbb E \frac { (a\cdot \xi)^2 (a\cdot e_1)^2} {\epsilon+(a\cdot e_1)^2}
\chi_{|a\cdot \xi|\ge N}  \le \mathbb E (a\cdot \xi)^2 \chi_{|a\cdot \xi|\ge N} \le
\mathbb E X^2 \chi_{|X|\ge N} \to 0,
\]
as $N$ tend to infinity. Thus the desired inequality easily follows.
\end{proof}

\begin{lem} \label{lemSe7_1a}
Let $0<\eta_0\ll 1$ be given. Then if $m\gtrsim n$, then the following
hold with high probability:
\begin{align*}
\frac 1m \sum_{k=1}^m \frac {(a_k\cdot \hat u)^4}
{(a_k\cdot e_1)^2} \ge 100, \qquad\forall\,
\hat u \in \mathbb S^{n-1} 
\text{ with $||\hat u \cdot e_1| -1| \ge \eta_0 $}.
\end{align*}
\end{lem}
\begin{proof}
Without loss of generality we write
\begin{align*}
\hat u = s e^{\perp} \pm \sqrt{1-s^2} e_1, \qquad e^{\perp} \in
\mathbb S^{n-1} \text{ with $e^{\perp} \cdot e_1=0$}.
\end{align*}
Clearly $|s| \ge s_0=s_0(\eta_0)>0$ where $s_0(\eta_0)$ is a constant depending only
on $\eta_0$. Take $a\sim \mathcal N(0, \operatorname{I_n})$ and observe that
\begin{align*}
&\mathbb E \frac { (a\cdot \hat u)^4} {
\epsilon(1+ (a\cdot e^{\perp})^2) + (a\cdot e_1)^2} \notag \\
\ge & \; \mathbb E \frac {s^4 (a\cdot e^{\perp} )^4} {
\epsilon (1+(a\cdot e^{\perp} )^2) + (a\cdot e_1)^2 } \notag \\
\ge & s_0^4\frac 1 {2\pi} \int_{1\le y \le 2, \,x \in \mathbb R}
\frac { y^4}{ \epsilon (1+y^2) +x^2}
e^{-\frac {x^2+y^2} 2} dx dy \notag \\
\ge & s_0^4 \frac 1 {200}
\int_{|x| \le 1} \frac 1 {5 \epsilon+ x^2} dx 
\ge s_0^4 \cdot O(\epsilon^{-\frac 12} ) \ge 200,
\end{align*}
if $\epsilon>0$ is taken sufficiently small. 
Now we fix this $\epsilon$. Clearly for $m\gtrsim n$ with high probability it holds
that
\begin{align*}
\frac 1m \sum_{k=1}^m \frac {(a_k\cdot \hat u)^4}
{(a_k\cdot e_1)^2}
& \ge \; \frac 1m \sum_{k=1}^m
\frac {(a_k\cdot \hat u)^4} { \epsilon(1+(a_k\cdot e^{\perp})^2) +(a_k\cdot e_1)^2}
\notag \\
& \ge\, 100, \qquad \forall\, \hat u \in \mathbb S^{n-1} \text{ with $||\hat u\cdot e_1|-1|\le
\eta_0$}.
\end{align*}
\end{proof}

\section{Technical estimates for Section \ref{S:model4b}}
\begin{lem} \label{lemSeH:0}
For any $\epsilon>0$, there exists $R_0=R_0(\beta,\epsilon)>0$ sufficiently small,
such that if $m\gtrsim n$, then the following hold with high probability:
\begin{align*}
 \frac Rm \sum_{k=1}^m 
\frac {(a_k\cdot \hat u)^2} 
{\beta R + (a_k\cdot e_1)^2} < \epsilon,
\qquad\forall\, \hat u \in \mathbb S^{n-1}, \quad\forall\, 0<R \le R_0.
\end{align*}
\end{lem}
\begin{proof}
Let $\phi \in C_c^{\infty}(\mathbb R)$ be such that $0\le \phi(x) \le 1$ for all $x$,
$\phi(x)=1$ for $|x|\le 1$ and $\phi(x)=0$ for $|x| \ge 2$. We then split the sum
as
\begin{align*}
 \frac  Rm \sum_{k=1}^m 
\frac {(a_k\cdot \hat u)^2} 
{\beta R + (a_k\cdot e_1)^2}  \le &   \frac 1 {\beta m }\sum_{k=1}^m
(a_k\cdot \hat u)^2 \phi( \frac {a_k \cdot e_1} {\eta_0} ) 
+   R \cdot \eta_0^{-2}  \frac 1m \sum_{k=1}^m (a_k\cdot \hat u)^2.
\end{align*}
Clearly the first term is amenable to union bounds, and we can make it sufficiently
small with high probability by taking $\eta_0$ small (depending only on $\beta$ and
$\epsilon$). The second term is trivial since we can take $R$ sufficiently small.
Thus we complete the proof.
\end{proof}

\begin{lem} \label{lemSeH:1}
There exists $R_1=R_1(\beta)>0$ sufficiently small,
such that if $m\gtrsim n$, then the following hold with high probability:
\begin{align*}
c_1\le \frac 1m \sum_{k=1}^m \frac {(a_k\cdot \hat u) (a_k\cdot e_1)^3}
{\beta R+ (a_k\cdot e_1)^2} \le c_2, 
\qquad\forall\, \hat u \in \mathbb S^{n-1} \text{ with $\widehat{u}_1\cdot
e_1 \ge \frac 1{10}$}, \quad\forall\, 0<R \le R_1.
\end{align*}
In the above $c_1$, $c_2>0$ are constants depending only on $\beta$.
\end{lem}
\begin{proof}
Denote $X_k= a_k\cdot e_1$. Write $\hat u = s e_1 +\sqrt{1-s^2} e^{\perp}$,
where $s\ge \frac 1 {10}$ and $ e^{\perp} \in \mathbb S^{n-1}$ satisfies
$e^{\perp}\cdot e_1=0$. We then write
\begin{align*}
 & \frac 1m \sum_{k=1}^m \frac { (a_k\cdot \hat u) X_k^3} {\beta R+ X_k^2}
 \notag \\
=& s \frac 1m
\sum_{k=1}^m \frac {X_k^4}{ \beta R+ X_k^2}
+ \sqrt{1-s^2} \frac 1m \sum_{k=1}^m
\frac { (a_k\cdot e^{\perp} ) X_k^3} { \beta R+ X_k^2} \notag \\
=&  s \frac 1m
\sum_{k=1}^m \frac {X_k^4}{ \beta R+ X_k^2}
+ \sqrt{1-s^2}\frac 1m \sum_{k=1}^m (a_k\cdot e^{\perp} ) X_k
- \sqrt{1-s^2} \frac 1m \sum_{k=1}^m (a_k\cdot e^{\perp})
\frac { \beta R X_k} {\beta R +X_k^2}. 
\end{align*}
For the first term we note that for $0<R\le 1$, 
\begin{align*}
\frac 1m \sum_{k=1}^m
\frac {X_k^4}{\beta+X_k^2}
\le \frac 1m \sum_{k=1}^m
\frac {X_k^4}{\beta R +X_k^2}
\le \frac 1m \sum_{k=1}^m X_k^2.
\end{align*}
Thus we 
clearly have for all $0<R\le 1$, $\frac 1 {10} \le s \le 1$, 
with high probability it holds that 
\begin{align*}
2c_1 \le s \frac 1m \sum_{k=1}^m \frac {X_k^4}{ \beta R+ X_k^2}
\le \frac 12 c_2. 
\end{align*}
The second term is clearly OK for union bounds and with high probability
it can be made sufficiently small.
For the last term,  observe that with high probability,
\begin{align*}
\frac 1m \sum_{k=1}^m |a_k\cdot e^{\perp}|
\frac { \beta R |X_k|} {\beta R +X_k^2}
\lesssim \sqrt{\beta R} \frac 1m \sum_{k=1}^m |a_k\cdot e^{\perp} | \ll 1,
\quad \forall\, e^{\perp} \in \mathbb S^{n-1},
\end{align*}
if $R \le R_1$ and $R_1$ is sufficiently small.
The desired result then clearly follows.
\end{proof}

\begin{proof}[Proof of Lemma \ref{Sep19_e0}]
Without loss of generality we consider the situation 
$\hat u = e_1 \cos \theta + e^{\perp} \sin \theta$ with
$\epsilon_1 \le \theta \le \frac {\pi}2- \epsilon_2$, where $0<\epsilon_1, \epsilon_2\ll 1$.
The point is that $\theta$ stays away from the end-points $0$ and $\frac {\pi}2$.
Denote $X_k=a_k \cdot e_1$, $Y_k = a_k \cdot e^{\perp}$ and $Z_k =a_k \cdot \hat u$.
Then 
\begin{align*}
& Z_k = \cos \theta X_k + \sin \theta Y_k \; \Rightarrow
\; Y_k = \frac 1 {\sin\theta} Z_k - \frac {\cos\theta} {\sin \theta} X_k ; \\
& \partial_{\theta} Z_k
= -\sin\theta X_k + \cos\theta Y_k
= \cot \theta Z_k - \frac 1 {\sin\theta} X_k.
\end{align*}
We then obtain
\begin{align*}
\partial_{\theta} f
&= 4R^2 \cot \theta \underbrace{\frac 1m \sum_{k=1}^m \frac {Z_k^4} {\beta R+ X_k^2}
}_{=:H_0}
-4R^2 \csc \theta \frac 1m \sum_{k=1}^m \frac {Z_k^3 X_k}{\beta R+ X_k^2}
\notag \\
& \qquad  - 4R \cot \theta \frac 1m \sum_{k=1}^m
\frac {Z_k^2 X_k^2} { \beta R+ X_k^2} 
+4R \csc \theta \frac 1m \sum_{k=1}^m
\frac { Z_k X_k^3} {\beta R+ X_k^2}.
\end{align*}
Since $R\sim 1$, it is not difficult to check that the third and fourth terms above are
amenable to union bounds\footnote{The union bound includes
covering in $\hat u$ and $R$.}, i.e. with high probability (for $m\gtrsim n$) we have
\begin{align*}
\Bigl| \frac 1m \sum_{k=1}^m
\frac {Z_k^2 X_k^2} { \beta R+ X_k^2} 
- \operatorname{mean} \Bigr| 
+ 
\Bigl| 
\frac 1m \sum_{k=1}^m
\frac { Z_k X_k^3} {\beta R+ X_k^2}
- \operatorname{mean} \Bigr|
\ll 1, 
\qquad \forall\, c_1 \le R\le c_2, \forall\, \hat  u \in \mathbb S^{n-1}.
\end{align*}

Next we treat the second term. 
Let $\phi \in C_c^{\infty}(\mathbb R)$ be such that $0\le \phi(x) \le 1$ for all $x$,
$\phi(x)=1$ for $|x|\le 1$ and $\phi(x)=0$ for $|x| \ge 2$.  We have
\begin{align*}
\frac 1m \sum_{k=1}^m
\frac {Z_k^3 X_k}
{\beta R+ X_k^2}
=\underbrace{\frac 1m \sum_{k=1}^m
\frac {Z_k^3 X_k}
{\beta R+ X_k^2}
\phi( \frac {Z_k} {M \langle X_k \rangle} ) }_{=:H_1}
+
\underbrace{\frac 1m \sum_{k=1}^m
\frac {Z_k^3 X_k}
{\beta R+ X_k^2}
\Bigl(1-\phi( \frac {Z_k} {M \langle X_k \rangle} ) \Bigr) }_{=:H_2},
\end{align*}
where $\langle z\rangle =(1+|z|^2)^{\frac 12}$.   It is not difficult to
check that $H_1$ is OK for union bounds, and
with high probability it holds that
\begin{align*}
\Bigl| H_1 - \mathbb E H_1 \Bigr| \ll 1,
\qquad\forall\, \hat u \in \mathbb S^{n-1}, \;
\forall\, c_1 \le R \le c_2.
\end{align*}
For $H_2$ we have ($\eta_0$ will be taken sufficiently small)
\begin{align*}
H_2 & \le 
 \eta_0 \frac 1m \sum_{k=1}^m
\frac {Z_k^4} { \beta R +X_k^2}
+ \eta_0^{-3} 
\frac 1m \sum_{k=1}^m  \frac{ X_k^4} {\beta R+ X_k^2} 
\Bigl(1-\phi( \frac {Z_k} {M \langle X_k \rangle} ) \Bigr) \notag \\
& \le \underbrace{\eta_0 \frac 1m \sum_{k=1}^m
\frac {Z_k^4} { \beta R+X_k^2}}_{=: H_{2,a}}
+ \underbrace{\eta_0^{-3} 
\frac 1m \sum_{k=1}^m  X_k^2 \Bigl(1-\phi( \frac {Z_k} {M \langle X_k \rangle} ) \Bigr).
}_{=:H_{2,b}} \notag \\
\end{align*}
We first take $\eta_0$ sufficiently small so that $H_{2,a}$ can be included in the estimate
of $H_0$ without affecting too much the main order.  On the other hand, once $\eta_0$ is fixed,
we can take $M$ sufficiently large such that
\begin{align*}
| H_{2,b} | +|\mathbb E H_{2,b}| \ll 1, \qquad\forall\, \hat u \in \mathbb S^{n-1},
\; \forall\, c_1 \le R\le c_2.
\end{align*}

Finally we treat $H_0$.  Clearly
\begin{align*}
H_0 \ge \underbrace{ \frac 1m \sum_{k=1}^m
\frac {Z_k^4} {\beta R+ X_k^2}
\phi( \frac {Z_k} { K  }).}_{=: H_{0,a}}
\end{align*}
By taking $K$ large, it can be easily checked that
\begin{align*}
\sup_{\hat u \in \mathbb S^{n-1},
c_1 \le R \le c_2} |\mathbb E H_0 - \mathbb E H_{0,a} | \ll 1.
\end{align*}
On the other hand, for fixed $K$, clearly
 $H_{0,a}$ is OK for union bounds. 
 It holds with high probability  that
\begin{align*}
| H_{0,a} -\mathbb E H_{0,a} | \ll 1.
\end{align*}
Collecting all the estimates, we obtain
\begin{align*}
\partial_{\theta} f \ge \mathbb E \partial_{\theta} f + \operatorname{Error},
\end{align*}
where $|\operatorname{Error} | \ll 1$. 
The desired lower bound for $\partial_{\theta} f$  then easily follows from Lemma 
\ref{lemSeH:3} below. 
\end{proof}
\begin{lem} \label{lemSeH:3}
Let $u=\sqrt R \hat u$ with $0<c_1\le R \le c_2 <\infty$ and $\hat u \in \mathbb S^{n-1}$.
Assume $\hat u = \cos \theta e_1+ \sin \theta e^{\perp}$, where
$\theta \in [0, {\pi}]$ and $e^{\perp} \in \mathbb S^{n-1}$ satisfies
$e^{\perp} \cdot e_1=0$. We have
\begin{align*}
&\mathbb E f(u) = h(\beta, R, \cos^2 \theta),
\end{align*}
where 
\begin{align*}
& \max_{0\le s\le 1}\partial_s h(\beta ,R ,s) \le -\gamma_{1} <0,  \\
& \min_{0\le s\le 1}\partial_{ss} h(\beta, R, s)\ge \gamma_2>0.
\end{align*}
Here $\gamma_i=\gamma_i(\beta, c_1,c_2)$, $i=1,2$ depend only on 
($\beta$, $c_1$, $c_2$).
It follows that
\begin{align*}
&\mathbb E \partial_{\theta} f =  a_1(\beta,R,\cos^2\theta) \sin (2\theta); \\
&\mathbb E \partial_{\theta \theta} f =  2a_1(\beta,R,\cos^2\theta) \cos  (2\theta)
+a_2(\beta,R,\theta) \sin^2(2\theta),
\end{align*}
where  
\begin{align*}
\gamma_3< a_i(\beta, R ,s ) \le \gamma_4, \forall\, s\in [0,1],\, i=1,2;
\end{align*}
and $\gamma_3>0$, $\gamma_4>0$ are constants depending only 
on ($\beta$, $c_1$, $c_2$).
\end{lem}
\begin{proof}
We have
\begin{align*}
\mathbb E f (u)
&= \frac 1 {2\pi} \int_{\mathbb R^2}
\frac { ( R ( x \cos \theta + y \sin \theta)^2 - x^2)^2}
{ \beta R + x^2} e^{-\frac {x^2+y^2} 2} dx dy \notag \\
&= \frac 1 {\pi} \int_0^{\infty} 
\frac 1{\beta R+x^2}
e^{-\frac {x^2} 2 } \cdot \sqrt{2\pi} h_1(R, x,\cos^2\theta) dx,
\end{align*}
where
\begin{align*}
h_1(R,x,s)= 3 R^2 -2R x^2+x^4+s(6R^2-2Rx^2)(-1+x^2)
+R^2 s^2(3-6x^2+x^4)
\end{align*}
Integrating further in $x$ then gives 
\begin{align*}
\mathbb E f(u) =\sqrt{2\pi}
\cdot \frac 1 {\pi}
\cdot R \Bigl( c_1 s^2 +2 c_2 s +c_3\Bigr),\qquad s=\cos^2\theta, 
\end{align*}
where the value of $c_3$ is unimportant for us, and
\begin{align*}
&c_1= R \int_0^{\infty} \frac 1 {\beta R+x^2} e^{-\frac {x^2}2}(3-6x^2+x^4) dx;\notag \\
&c_2=  \int_0^{\infty}
\frac 1 {\beta R +x^2}e^{-\frac {x^2}2} (3R-x^2)(-1+x^2) dx.
\end{align*}
First we show that $c_2<0$. By a short computation, we have
\begin{align*}
c_2 =\frac {3+\beta}{2\beta}
\cdot \Bigl( \beta R \sqrt{2\pi} 
-e^{\frac {\beta R}2} \pi 
\sqrt{\beta R} (1+\beta R)
\operatorname{Erfc}( \sqrt{ \frac {\beta R} 2} ) \Bigr),
\end{align*}
where
\begin{align*}
\operatorname{Erfc}(y)= \frac 2 {\sqrt{\pi} } \int_y^{\infty} e^{-t^2} dt.
\end{align*}
We then reduce the matter to showing
\begin{align} \label{lemH3.001tmp}
y < e^{y^2} (1+2y^2) \int_y^{\infty} e^{-t^2} dt, \quad\forall\, y>0.
\end{align}
This follows easily from the usual bound on $\operatorname{Erfc}(y)$:
\begin{align} \label{lemH3.002tmp}
\frac 1 {y+\sqrt{y^2+2}}
<\operatorname{Erfc}(y) \cdot e^{y^2}\cdot \frac {\sqrt{\pi}}2 \le 
\frac 1 {y+\sqrt{y^2+\frac 4 {\pi} } }, \quad \forall\, y>0.
\end{align}
Thus $c_2<0$. 

Next we show that $c_1>0$.  We have
\begin{align*}
2\beta  c_1=
-\sqrt{2\pi} \beta R (5+\beta R)
+e^{\frac {\beta R}2} \pi
\sqrt{\beta R}
(3+\beta R (6+\beta R) ) \cdot \operatorname{Erfc}(
\sqrt{\frac {\beta R} 2} ).
\end{align*}
It amounts to checking
\begin{align*}
e^{y^2} \int_y^{\infty}
e^{-t^2} dt >\frac {y (5+2y^2)}
{3+4y^2(3+y^2)}, \quad\forall\, y>0.
\end{align*}
This follows from Lemma \ref{lemH3.ad01} below.

Finally we show $c_1+c_2<0$.  We have
\begin{align*}
&2(c_1+c_2) \notag \\
=&\;
\sqrt{2\pi} R (-2+\beta -\beta R)
-e^{\frac {\beta R}2}
\pi\cdot(-\beta^{\frac 32} R^{\frac 52}
+ (\beta R)^{\frac 32}+\sqrt{\beta R}
-3R \sqrt{\beta R})
\operatorname{Erfc}(\sqrt{\frac {\beta R}2}).
\end{align*}
Denote $y=\sqrt{\frac {\beta R}2}>0$. We then reduce matters to showing
\begin{align*}
2y^2 -2R (1+y^2)
<e^{y^2} \cdot 2y \cdot (-2y^2 R-3R +1+2y^2) \int_y^{\infty} e^{-t^2}dt.
\end{align*}
Since we have shown \eqref{lemH3.001tmp}, we then only need to check
\begin{align*}
1+y^2>e^{y^2} y (2y^2+3) \int_y^{\infty} e^{-t^2} dt.
\end{align*}
This in turn follows from Lemma \ref{lemH3.ad01}.

Finally we consider the polynomial
\begin{align*}
\tilde h(s) = c_1 s^2+2c_2 s.
\end{align*}
Since $\tilde h^{\prime}(s) = 2c_1 s+2c_2$ and $\tilde h^{\prime}(0)=2c_2<0$,
$\tilde h^{\prime}(1)=2c_1+2c_2<0$, we have $\tilde h^{\prime}(s)<0$ for all
$s\in[0,1]$. Since $c_1>0$, we have $\tilde h^{\prime\prime}(s)>0$. The desired
result then easily follows.
\end{proof}

\begin{lem} [Refined upper and lower bounds on the
Complementary Error function]\label{lemH3.ad01}
Let $\operatorname{Erfc}(x) = \frac 2 {\sqrt{\pi}}
\int_x^{\infty} e^{-t^2} dt$ for $x>0$. Then 
\begin{align*}
&e^{x^2} \cdot \operatorname{Erfc}(x) \cdot 
\frac {\sqrt{\pi}} 2
> \frac {x(5+2x^2)} {3+4x^2(3+x^2)}, \quad\forall\, x>0; \\
&e^{x^2} \cdot \operatorname{Erfc}(x) \cdot 
\frac {\sqrt{\pi}} 2
< \frac {1+x^2}{x(3+2x^2)}, \quad\forall\, x>0.
\end{align*}
\end{lem}
\begin{rem}
In the regime $y\ge 1$, one can check that the upper and lower bounds here are sharper than
\eqref{lemH3.002tmp}. 
One should also recall that the usual way
to derive the lower bound in \eqref{lemH3.002tmp} through
conditional expectation. Namely one can regard
$e^{-y^2}/(\sqrt{\pi} \operatorname{Erfc}(y) )$ as the conditional
mean $\mu_1(y)= \mathbb E(X|X>y)$ where $X$ has the p.d.f. 
$\frac 1 {\sqrt{\pi} } e^{-x^2}$. Then evaluating the variance
$\mathbb E ( (X-\mu_1)^2|X>y)>0$ gives $y\mu_1+\frac 12 -\mu_1^2>0$.
This  yields the upper bound for $\mu_1$ which in turn is the desired lower bound
in \eqref{lemH3.002tmp}. An interesting question is to derive a sharper two-sided
bounds via more careful conditioning. However we shall not dwell on this issue here.
\end{rem}
\begin{proof}[Proof of Lemma \ref{lemH3.ad01}]
We focus on the regime $x> 1$. 
By performing successive simple change of variables, we have
\begin{align*}
g(x):=e^{x^2} \int_x^{\infty}
e^{-t^2} dt 
& = \int_0^{\infty} e^{-2x s} e^{-s^2} ds \notag \\
&=\frac 1 {2x} \int_0^{\infty}
e^{-s} e^{- (\frac {s} {2x} )^2} ds \notag \\
&\sim \sum_{k=0}^{\infty} (-1)^k x^{-(2k+1)} \cdot 
\frac 12 \cdot \frac {(2k)!} {4^k k!} \notag \\
&\sim \sum_{k=0}^{\infty} (-1)^k x^{-(2k+1)} \cdot 
\frac 12 \cdot  \left(\frac 12\right)_k,
\end{align*}
where in the last line we adopted Pochhammer's symbol $(a)_n
=a(a+1)\cdots(a+n-1)$.  Note that the above is an asymptotic series, and it 
is not difficult to check that 
\begin{align*}
\Bigl|
g(x) - \sum_{k=0}^m
(-1)^k x^{-(2k+1)} \cdot 
\frac 12
\cdot \left( \frac 12 \right)_k
\Bigr|\le x^{-2m-3}\cdot \frac 12 \cdot 
\left( \frac 12 \right)_{m+1}, \qquad\forall\, m\ge 1, \forall\, x>0.
\end{align*}
Moreover, if $m$ is an even integer, then
\begin{align*}
g(x) <\sum_{k=0}^m
(-1)^k x^{-(2k+1)} \cdot 
\frac 12
\cdot \left( \frac 12 \right)_k, \qquad\forall\, x>0;
\end{align*}
and if $m$ is odd, then
\begin{align*}
g(x) >\sum_{k=0}^m
(-1)^k x^{-(2k+1)} \cdot 
\frac 12
\cdot \left( \frac 12 \right)_k, \qquad\forall\, x>0.
\end{align*}

Now taking $m=4$, we have
\begin{align*}
g(x)< \frac 12 x^{-1} -\frac 14 x^{-3}
+\frac 38 x^{-5} -\frac {15}{16} x^{-7}
+\frac{105}{32} x^{-9}.
\end{align*}
For $x\ge 3$, it is not difficult to verify that 
\begin{align*}
\frac 12 x^{-1} -\frac 14 x^{-3}
+\frac 38 x^{-5} -\frac {15}{16} x^{-7}
+\frac{105}{32} x^{-9} < \frac{1+x^2}{x(3+2x^2)}.
\end{align*}
Hence the upper bound is OK for $x\ge 3$. 

Next taking $m=5$, we have
\begin{align*}
g(x)> \frac 12 x^{-1} -\frac 14 x^{-3}
+\frac 38 x^{-5} -\frac {15}{16} x^{-7}
+\frac{105}{32}x^{-9}-\frac{945}{64}x^{-11}.
\end{align*}
It is not difficult to verify that for $x\ge 4$, we have
\begin{align*}
\frac 12 x^{-1} -\frac 14 x^{-3}
+\frac 38 x^{-5} -\frac {15}{16} x^{-7}
+\frac{105}{32}x^{-9}-\frac{945}{64}x^{-11}
>\frac{x(5+2x^2)}{3+12x^2+4x^4}.
\end{align*}
Hence the lower bound is OK for $x\ge 4$. 

Finally for the regime $x\in[0,4]$, we use rigorous numerics to verify the inequality.
Since we are on a compact interval, this can be done by a rigorous computation
with controllable numerical errors. 
\end{proof}


\begin{proof}[Proof of Lemma \ref{Sep19_e1}]
Again denote $X_k=a_k\cdot e_1$ and $Z_k=a_k\cdot \hat u$.
Without loss of generality we assume $\theta \in [\frac {\pi}2 -\eta,
\frac {\pi} 2+\eta]$ for some sufficiently small $\eta>0$. 
By a tedious computation, we have 
\begin{align*}
\partial_{\theta\theta} f
& = 4R^2 (1+2 \cos 2 \theta) \csc^2 \theta \frac 1m 
\sum_{k=1}^m \frac {Z_k^4} {\beta R+ X_k^2} 
- 24R^2 (\cot \theta \csc \theta) \frac 1m \sum_{k=1}^m
\frac { X_k Z_k^3} {\beta R+ X_k^2} \notag \\
& \quad +4R (\csc^2 \theta) (3R-\cos 2\theta) \frac 1m \sum_{k=1}^m
\frac{ Z_k^2 X_k^2} { \beta R+ X_k^2} 
+8 R (\cot \theta \csc \theta)  \frac 1m \sum_{k=1}^m
\frac { X_k^3 Z_k} { \beta R + X_k^2} \notag \\
& \qquad\qquad
-4R \csc^2 \theta\frac 1m \sum_{k=1}^m \frac {X_k^4}{\beta R+ X_k^2}.
\end{align*} 
Note that the third, fourth and fifth terms are OK for union bounds. The second and the
first term
can be handled in a similar way as in the proof of  Lemma \ref{Sep19_e0}. The only difference
is that the sign is now negative in the regime $\theta \to \frac {\pi}2$.  Using Lemma
\ref{lemSeH:3} it follows that $\partial_{\theta\theta} f <0$ in this regime. 
We omit the
repetitive details.
\end{proof}

\begin{proof}[Proof of Theorem \ref{Sep19e8}]
Without loss of generality we consider the regime $\| u-e_1\|_2 \ll 1$. 
Before we work out the needed estimates for the restricted convexity, we explain
the main difficulty in connection with the full Hessian matrix. 
Denote $X_k =a_k \cdot e_1$.  Then for any $\xi \in \mathbb S^{n-1}$, we have
\begin{align}
H_{\xi\xi} & = \sum_{i,j} \xi_i \xi_j  (\partial_{u_i u_j} f)(u) \notag \\
&= 12 \frac 1m \sum_{k=1}^m \frac {(a_k\cdot \xi)^2 (a_k \cdot u)^2}{
\beta |u|^2 + X_k^2}    \label{Sep19e8.1a} \\
& \quad-4 \frac 1m \sum_{k=1}^m \frac { (a_k \cdot \xi)^2 X_k^2} { \beta |u|^2+
X_k^2}  \label{Sep19e8.1b} \\
&\quad -16 \beta \frac 1 m\sum_{k=1}^m
\frac { (a_k\cdot u)^3 (a_k\cdot \xi ) (u \cdot \xi) } { (\beta |u|^2 + X_k^2)^2}
\label{Sep19e8.1c} \\
& \quad + 16\beta \frac 1m
\sum_{k=1}^m\frac {  X_k^2 (a_k\cdot u) (a_k\cdot \xi) (u\cdot \xi)}
{(\beta |u|^2 + X_k^2)^2} \label{Sep19e8.1d} \\
& \quad- 2\beta \frac 1m \sum_{k=1}^m 
\frac { ( (a_k\cdot u)^2 -X_k^2)^2} { (\beta |u|^2 + X_k^2)^2}  
\label{Sep19e8.1e} \\
& \quad +8\beta^2 (\xi \cdot u)^2
\frac 1m \sum_{k=1}^m 
\frac { ( (a_k\cdot u)^2 -X_k^2)^2} { (\beta |u|^2 + X_k^2)^3}.
\label{Sep19e8.1f}
\end{align}

First observe that if $u=e_1$, then the Hessian can be controlled rather easily thanks
to the damping $\beta |u|^2+ X_k^2$.

On the other hand, for $u\ne e_1$, 
as far as the lower bound is concerned, 
the main difficult terms are \eqref{Sep19e8.1e} and
\eqref{Sep19e8.1c} which are out of control if we do not impose any
condition on $\xi$ (i.e. using \eqref{Sep19e8.1a} to control it). On the other hand,
if we restrict $\xi$ to the direction $ u -e_1$, then we can control these difficult terms by using
the main good term \eqref{Sep19e8.1a}. Namely, introduce the decomposition
\begin{align*}
u= e_1 +t \xi,  
\end{align*}
where $t= \| u-e_1\|_2 \ll 1$.  Then for  \eqref{Sep19e8.1c} we write 
\begin{align*}
(a_k\cdot u)^3 (a_k\cdot \xi) = (a_k\cdot u)^2 (a_k\cdot e_1) (a_k\cdot \xi)
+ t (a_k\cdot u)^2 (a_k\cdot \xi)^2
\end{align*}
Since $t \ll 1$, the term $t (a_k\cdot u)^2 (a_k\cdot e_1)^2$ (together with the
pre-factor term in \eqref{Sep19e8.1c})  can be included into 
\eqref{Sep19e8.1a} which still has a good lower bound by using localization. 
On the other hand, the term $(a_k\cdot u)^2 (a_k\cdot e_1) (a_k\cdot \xi)$
can be split as
\begin{align}
&(a_k\cdot u)^2 (a_k\cdot e_1) (a_k\cdot \xi) \notag \\
= &(a_k\cdot u)^2 (a_k\cdot e_1) (a_k\cdot \xi) \phi( \frac {a_k \cdot u} K) \label{Sp19e8.2a}
\\
& \quad + (a_k\cdot u)^2 (a_k\cdot e_1) (a_k\cdot \xi) \Bigl( 1-
\phi(\frac {a_k\cdot u} K) \Bigr), \label{Sp19e8.2b}
\end{align}
where $\phi$ is a smooth cut-off function satisfying $0\le \phi(z)\le 1$ for all $z\in \mathbb R$,
$\phi(z)=1$ for $|z|\le 1$ and $\phi(z)=0$ for $|z|\ge 2$. Clearly the contribution of
\eqref{Sp19e8.2a} in \eqref{Sep19e8.1c} is OK for union bounds. On the other hand,
for \eqref{Sp19e8.2b} we have
\begin{align*}
&(a_k\cdot u)^2 |a_k\cdot e_1| |a_k\cdot \xi| \cdot \Bigl( 1-
\phi(\frac {a_k\cdot u} K) \Bigr) \notag \\
\le &\;
(a_k\cdot u)^2 \epsilon (a_k\cdot \xi)^2 
+\epsilon^{-1} (a_k\cdot u)^2 (a_k\cdot e_1)^2 
\Bigl( 1- \phi( \frac {a_k\cdot u} K) \Bigr).
\end{align*}
Clearly this is under control (the first term can again be controlled using \eqref{Sep19e8.1a}). 

Now we turn to \eqref{Sep19e8.1e}. The main term is $(a_k\cdot u)^4$.  We write
\begin{align*}
(a_k\cdot u)^2 (a_k\cdot u)^2
= (a_k\cdot u)^2 ( a_k\cdot e_1)^2 
+ t^2 (a_k\cdot u)^2  (a_k\cdot \xi)^2
+ 2t (a_k\cdot u)^2 (a_k\cdot e_1) (a_k\cdot \xi).
\end{align*}
Clearly then this is also under control. 

By further using localization, we can then show that with high probability,
it holds that 
\begin{align*}
H_{\xi\xi} \ge \mathbb E H_{\xi \xi} + \operatorname{Error},
\end{align*}
where $|\operatorname{Error}| \ll 1$.  The desired conclusion then follows
from Lemma \ref{Sp20.1}.
\end{proof}
\begin{rem}
Introduce the parametrization $u =\sqrt R ( e_1 \cos \theta + e^{\perp} \sin \theta)$ where
$e^{\perp} \in e_1=0$, $|R-1| \ll1 $ and $ |\theta| \ll 1$. One might hope to prove that
the Hessian matrix 
\begin{align*}
\begin{pmatrix}
\partial_{RR} f \quad \partial_{R\theta} f \\
\partial_{R\theta}f \quad \partial_{\theta\theta} f
\end{pmatrix}
\end{align*}
is positive definite near $u=e_1$ under the mere assume $m\gtrsim n$ and with
high probability. However there is a subtle issue which we explain as follows.
Consider the main term  (write $X=a_k\cdot e_1$ and $Y=a_k\cdot e^{\perp}$)
\begin{align*}
\tilde f=\tilde f_k = \frac { \Bigl( R (X\cos \theta + Y \sin \theta)^2 -X^2 \Bigr)^2}
{\beta R+ X^2}.
\end{align*}
The most troublesome piece come from quartic and cubic terms in $Y$, and we consider
\begin{align*}
\tilde h_1= \frac {R^2  Y^4 \sin^4 \theta  }
{\beta R+X^2},
\quad \tilde h_2=
 \frac {R^2 \Bigl(  4 Y^3 X\sin^3 \theta \cos \theta \Bigr)}
{\beta R+X^2}.
\end{align*}
For $\tilde h_2$ we do not have a favorable sign and the only hope is to control it via $\tilde h_1$.
On the other hand, for $\tilde h_1$,  we can take $X=Y=1$, $\beta=1$, and compute
\begin{align*}
(\partial_{RR} \tilde h_1 ) \cdot (\partial_{\theta\theta} \tilde h_1)
- (\partial_{R\theta} \tilde h_1)^2
= - \frac {8R^2} {(1+R)^4}
\sin^6 \theta
\cdot \Bigl( 3+4R +R^2+ (2+4R+R^2) \cos 2\theta \Bigr).
\end{align*}
In yet other words, the sign is not favorable and this renders the Hessian out of control
(before taking the expectation).
\end{rem}

\begin{lem} \label{Sp20.1}
Let $u= e_1 +t \xi$ where $\xi \in \mathbb S^{n-1}$. Then for $|t| \ll 1$, we have
\begin{align*}
\mathbb E \partial_{tt} f(u) \ge c_0>0, \qquad\forall\, \xi \in \mathbb S^{n-1},
\end{align*}
where $c_0>0$ is a constant depending only on $\beta$.
\end{lem}
\begin{proof}[Proof of Lemma \ref{Sp20.1}]
Introduce the parametrization $\xi =se_1 +\sqrt{1-s^2} e^{\perp}$ 
where $e^{\perp} \cdot e_1=0$, $|s|\le 1$.  Then
$u=e_1+ t (se_1+ \sqrt{1-s^2} e^{\perp})=(1+ts)e_1+ t\sqrt{1-s^2} e^{\perp}$.
Thus
\begin{align*}
\mathbb E f(u)
= \frac 1 {2\pi}
\int_{\mathbb R^2}
\underbrace{\frac { \Bigl( ( (1+ts)x + t \sqrt{1-s^2} y)^2 -x^2 \Bigr)^2}
{\beta (1+2ts +t^2) +x^2} }_{=:h(t,s,x,y)} e^{-\frac {x^2+y^2} 2} dx dy.
\end{align*}
It is not difficult to check that
\begin{align*}
\partial_{tt} h(t,s,x,y)\Bigr|_{t=0}= 
\frac { 8 x^2 (s x + \sqrt{1-s^2} y)^2} {\beta +x^2}.
\end{align*}
Thus it follows that
 \begin{align*}
 \mathbb E \partial_{tt} f(u) \Bigr|_{t=0,|s|\le 1} \gtrsim 1.
 \end{align*}
 The desired result then follows by a simple perturbation argument using the fact
 that $\mathbb E\partial_{ttt}f$ is uniformly bounded and taking $|t|$ sufficiently
 small. 
\end{proof}

\section{Technical estimates for Section  \ref{S:model4c}}
\begin{lem} \label{lemSeI:1}
Let $u=\sqrt R \hat u$ with $0<c_1\le R \le c_2 <\infty$ and $\hat u \in \mathbb S^{n-1}$.
Assume $\hat u = \cos \theta e_1+ \sin \theta e^{\perp}$, where
$\theta \in [0, {\pi}]$ and $e^{\perp} \in \mathbb S^{n-1}$ satisfies
$e^{\perp} \cdot e_1=0$. We have
\begin{align*}
&\mathbb E \partial_{\theta} f =  a_1(\beta_1,\beta_2, R,\theta) \sin (2\theta); 
\end{align*}
where  
\begin{align*}
\gamma_1< a_1(\beta_1,\beta_2, R ,\theta ) \le \gamma_2, \quad\forall\, \theta\in [0,\pi],\,
c_1\le R\le c_2;
\end{align*}
and $\gamma_1>0$, $\gamma_2>0$ are constants depending only 
on ($\beta_1$, $\beta_2$, $c_1$, $c_2$).
Furthermore for some sufficiently small constants
$\theta_0=\theta_0(\beta_1,\beta_2,c_1,c_2)>0$, $\theta_1=\theta_1(\beta_1,
\beta_2,c_1,c_2)>0$, we have
\begin{align*}
 &\gamma_3 < \mathbb E \partial_{\theta\theta} f < \gamma_4,
 \qquad \text{if $0\le \theta \le \theta_0$ or $\pi -\theta_0 \le \theta \le \pi$}, \\
 & \gamma_5< - \mathbb E \partial_{\theta\theta} f <\gamma_6,
 \qquad \text{if $\Bigl|\theta-\frac {\pi}2 \Bigr| <\theta_1$},
 \end{align*}
 where $\gamma_i>0$, $i=3,\cdots, 6$ depend only on ($\beta_1$, $\beta_2$,
 $c_1$, $c_2$).
\end{lem}

\begin{proof}
We have
\begin{align*}
\mathbb E f(u)
= \frac 1 {2\pi}
\int_{\mathbb R^2}
\frac { \Bigl( R(x\cos \theta +y \sin \theta)^2 -x^2 \Bigr)^2}
{R + \beta_1R(x\cos \theta+y\sin \theta)^2+\beta_2 x^2}
e^{-\frac{x^2+y^2}2} 
dx dy.
\end{align*}
Denote 
\begin{align*}
h(a,b)= \frac {( R a^2 -b)^2} {R +\beta_1 R a^2 +\beta_2 b}.
\end{align*}
Then
\begin{align*}
& \partial_{\theta}
\Bigl( h(x\cos \theta +y \sin \theta, x^2) \Bigr)
= (-x \sin \theta +y \cos \theta) \partial_a h; \\
&\partial_x \Bigl( h (x\cos \theta+y \sin \theta, x^2) 
\Bigr) = \partial_a h \cdot \cos \theta + 2x \partial_b h;\\
& \partial_y \Bigl( h (x\cos \theta+y \sin \theta, x^2) 
\Bigr) = \partial_a h \cdot \sin \theta ;\\
& \partial_{\theta}
\Bigl( h(x\cos \theta +y \sin \theta, x^2) \Bigr)
= (y\partial_x- x\partial_y) 
\Bigl( h(x\cos \theta +y \sin \theta, x^2) \Bigr)
-2xy \partial_b h.
\end{align*}
By using integration by parts, we then obtain
\begin{align*}
\mathbb E \partial_{\theta} f 
&=\frac 1{\pi} \int_{\mathbb R^2}
(-xy) (\partial_b h)(x\cos\theta+y\sin \theta, x^2) 
e^{-\frac {x^2+y^2} 2} dx dy \notag \\
&=\frac 2 {\pi}
\int_{x>0, y>0}
\Bigl( (\partial_b h)(x\cos \theta-y \sin \theta,x^2)
-(\partial_b h)(x\cos \theta+y \sin \theta, x^2)
\Bigr) xy e^{-\frac {x^2+y^2}2} dx dy.
\end{align*}
Now denote
\begin{align*}
h_1(a,b)= \frac {( R a -b)^2} {R +\beta_1 R a+\beta_2 b}.
\end{align*}
It is not difficult to check that for
$a\ge 0$, $b\ge 0$, $\beta_1,\beta_2>0$, $R>0$,
\begin{align*}
\partial_{ab} h_1 =
- 2R^2 \frac { (1+a(\beta_1+\beta_2)) \cdot ( b (\beta_1+\beta_2)+R) }
{ (\beta_2 b +R + \beta_1 a R)^3}<0.
\end{align*}
Observe that 
\begin{align*}
(\partial_b h)(a,b)= (\partial_b h_1)(a^2,b).
\end{align*}
Then if $x,y>0$ and $\theta \in [0, {\pi}]$, then
\begin{align*}
& (\partial_b h)(x\cos \theta-y \sin \theta,x^2)
-(\partial_b h)(x\cos \theta+y \sin \theta, x^2) \notag \\
= &
(\partial_b h_1 )( (x\cos \theta-y \sin \theta)^2,x^2)
-(\partial_b h_1)( (x\cos \theta+y \sin \theta)^2, x^2)\notag \\
=&  - 2\int_0^1 (\partial_{ab} h_1) \Bigl(
(x\cos \theta+y \sin \theta)^2 -4\tau xy \cos\theta \sin \theta
, x^2\Bigr) d\tau \cdot xy \cdot \sin(2\theta).
\end{align*}
Integrating in $x$ and $y$, we then obtain
\begin{align*}
&\mathbb E \partial_{\theta} f =  a_1(\beta_1,\beta_2, R,\theta) \sin (2\theta),
\end{align*}
 where $a_1 \sim 1$ and is a smooth function of $\theta$. Differentiating in $\theta$ then
 gives
 \begin{align*}
&\mathbb E \partial_{\theta\theta} f =  2a_1(\beta_1,\beta_2, R,\theta) \cos (2\theta)
+ \partial_{\theta} a_1(\beta_1,\beta_2, R, \theta) \sin (2\theta).
\end{align*}
Then second term clearly vanishes near $\theta=0, \, \frac {\pi}2, \pi$. 
Thus the desired estimate for $\mathbb E \partial_{\theta\theta} f$ follows.
\end{proof}

\begin{lem}[Strong convexity of $\mathbb E f$ when $\| u\pm e_1\| \ll 1$]
\label{lemSeI:2}
Let $h(u)=\mathbb E f (u)$.
There exists $0<\epsilon_0\ll 1$ such that the following hold:
\begin{enumerate}
\item If $\| u-e_1\|_2 \le \epsilon_0$, then for any $\xi \in \mathbb S^{n-1}$, we have
\begin{align*}
\sum_{i,j=1}^n \xi_i \xi_j 
(\partial_i \partial_j h)(u) \ge \gamma_1 >0,
\end{align*}
where $\gamma_1$ is a constant. 

\item If $\| u+ e_1\|_2 \le \epsilon_0$, then for any $\xi \in \mathbb S^{n-1}$, we have
\begin{align*}
\sum_{i,j=1}^n \xi_i \xi_j 
(\partial_i \partial_j h)(u) \ge \gamma_1 >0.
\end{align*}
\end{enumerate}
\end{lem}
\begin{proof}
We shall employ the same approach as in the proof of Theorem 2.5 in the second paper of this series of work and
sketch only the needed modifications. Without loss of generality consider the 
regime $\| u-e_1\|_2\ll 1$ and introduce the change of variables:
\begin{align*}
& u = \rho  \hat u; \\
& \hat u = \sqrt{1-s^2} e_1 + s e^{\perp}, \qquad e^{\perp} \cdot e_1 =0, \, e^{\perp}
\in \mathbb S^{n-1}, 
\end{align*}
where $ |\rho -1 | \ll 1$ and $0\le s \ll 1$.  Denote
\begin{align*}
h_1(\rho,s) = h( u) =h( \rho ( \sqrt{1-s^2} e_1 +s e^{\perp}) ),
\end{align*}
where we note that the value of $h(u)$ depends only on ($\rho$, $s$). Clearly
\begin{align*}
h_1(\rho,s) = \frac 1 {2\pi}
\int_{\mathbb R^2}
\underbrace{\frac { \Bigl(\rho^2 ( \sqrt{1-s^2} x + s y)^2 - x^2 \Bigr)^2}
{ \rho^2 +\beta_1 \rho^2 (\sqrt{1-s^2} x +sy)^2  +\beta_2 x^2} }_{=:h_2(\rho,s,x,y)}
e^{-\frac {x^2+y^2} 2} dx dy.
\end{align*}
It is easy to check that 
\begin{align*}
\max_{\frac 12 \le \rho \le 2, |s|\le \frac 12}
\sum_{i,j\le 4} |\partial_{\rho}^i \partial_s^{j} h_1(\rho,s)| \lesssim 1.
\end{align*}
By a tedious computation, we have
\begin{align*}
&\partial_{\rho\rho}h_2(\rho,0,x,y) \notag \\
=&\;
\frac{ 2(3\rho^2+\rho^6) x^4 +k_1 \cdot x^6+ k_2 x^8}
{(\beta_2 x^2 +\rho^2(1+\beta_1 x^2) )^3},
\end{align*}
where 
\begin{align*}
&k_1=2(-\beta_2 +6\beta_1 \rho^2 +6\beta_2 \rho^2 +3\beta_2 \rho^4
+2\beta_1 \rho^6);\\
&k_2=2(-\beta_1\beta_2
-2\beta_2^2+3\beta_1^2\rho^2
+6\beta_1\beta_2 \rho^2 +6\beta_2^2 \rho^2 +3\beta_1\beta_2 \rho^4
+\beta_1^2 \rho^6).
\end{align*}
Since $\rho \to 1$, it is clear that $k_1>0$ and $k_2>0$, and thus 
\begin{align*}
\partial_{\rho\rho} h_1 (1, 0) \gtrsim 1.
\end{align*}
It is not difficult to check that $\partial_{s} h_1 (\rho, 0)=0$ for any $\rho>0$.
Clearly also $\partial_{\rho s} h_1( \rho, 0)=0$ for any $\rho>0$. 
To compute $\partial_{ss} h_1(1,0)$ we shall use Lemma \ref{lemSeI:1}. Observe
that ($s=\sin \theta$ with $\theta\to 0+$)
\begin{align*}
h_1(\rho, \sin \theta) &= \mathbb E f (u) ;\\
\cos \theta \partial_s h_1(\rho, \sin \theta) &= \mathbb E \partial_{\theta} f; \\
 -\sin \theta \partial_s h_1(\rho,\sin \theta)
+\cos^2\theta \partial_{ss} h_1(\rho,\sin \theta)&=
\mathbb E \partial_{\theta\theta} f.
\end{align*}
Clearly it follows that
\begin{align*}
\partial_{ss} h_1(1,0) \gtrsim 1.
\end{align*}
The rest of the argument is then essentially the same as in the proof of 
Theorem 2.5 in the second paper. We omit further details.
\end{proof}
\begin{proof}[Proof of Theorem \ref{Sep24e1}]
We rewrite
\begin{align*}
f(u)  =\frac 1m \sum_{k=1}^m G( |u|^2, (a_k\cdot u)^2, X_k^2),
\end{align*}
where 
\begin{align*}
G(a,b,c) = \frac { (b-c)^2} {a+\beta_1 b+\beta_2 c}.
\end{align*}
Clearly for any $\xi \in \mathbb S^{n-1}$, 
\begin{align} 
&\sum_{i,j=1}^n \xi_i \xi_j \partial_{u_i u_j} f \notag \\
=&\; \frac 1m \sum_{k=1}^m \partial_a G \cdot 2 |\xi|^2  \label{Sep24eA0.1}\\
& \quad + \frac 1m \sum_{k=1}^m \partial_{aa} G \cdot 4 (u\cdot \xi)^2 
\label{Sep24eA0.2} \\
& \quad + \frac 1m \sum_{k=1}^m \partial_{ab} G \cdot 8 (a_k\cdot u)
(a_k\cdot \xi) (\xi \cdot u)  \label{Sep24eA0.3}\\
& \quad+ \frac 1m \sum_{k=1}^m \partial_{bb} G
\cdot 4 (a_k\cdot u)^2 (a_k\cdot \xi)^2 \label{Sep24eA0.4} \\
&\quad+ \frac 1m \sum_{k=1}^m \partial_b G \cdot 2 (a_k\cdot \xi)^2.
\label{Sep24eA0.5}
\end{align}
In the above, $\partial_a G = (\partial_a G)( |u|^2, (a_k\cdot u)^2, X_k^2)$ and
similar notation is used for $\partial_{aa} G$, $\partial_{bb}G$, $\partial_b G$.

\texttt{Estimate of \eqref{Sep24eA0.1} and \eqref{Sep24eA0.2}}.  Clearly these two terms
are OK for union bounds, and we have (for $m\gtrsim n$ and with high probability)
\begin{align*}
|\eqref{Sep24eA0.1} - \operatorname{mean} |
+|\eqref{Sep24eA0.2} - \operatorname{mean} |
 \ll 1, \qquad\forall\, \xi \in
\mathbb S^{n-1}, \qquad\forall\, \frac 12 \le \|u\|_2 \le 2.
\end{align*}

\texttt{Estimate of \eqref{Sep24eA0.3}}. We have
\begin{align*}
(\partial_{ab} G)(a,b,c)= - \frac{2(b-c) (a+(\beta_1+\beta_2) c)} { (a+\beta_1 b+\beta_2 c)^3}.
\end{align*}
Consider the function 
\begin{align*}
\tilde G_1 (a,y,c) = - 
y \frac{2(y^2-c) (a+(\beta_1+\beta_2) c)} { (a+\beta_1 y^2+\beta_2 c)^3}.
\end{align*}
Clearly for $\frac 1 {10}\le a,\tilde a \le 10$, $y,\tilde y\in \mathbb R$, $c\ge 0$,
we have $|\tilde G_1|\lesssim 1$ and 
\begin{align*}
|\tilde G_1(a,y,c)-\tilde G_1(\tilde a,\tilde y, c)|
\lesssim |a-\tilde a| +|y-\tilde y|.
\end{align*}

Then for any ($u$, $\tilde u$) with
$\frac 12 \le \|u\|_2, \|\tilde u\|_2 \le 2$ and ($\xi$, $\tilde \xi$) with
$\xi, \tilde \xi \in \mathbb S^{n-1}$, we have 
\begin{align*}
  &\Bigl| (\partial_{ab} G) (|u|^2, (a_k\cdot u)^2, X_k^2) (a_k\cdot u) (a_k\cdot \xi) \notag \\
  &\quad- (\partial_{ab} G)(|\tilde u|^2,
  (a_k\cdot \tilde u)^2, X_k^2) (a_k\cdot \tilde u ) (a_k\cdot \tilde \xi) \Bigr| \notag \\
\lesssim & \;|a_k\cdot (\xi -\tilde \xi)| + |a_k\cdot \xi| \cdot 
(|a_k\cdot (u-\tilde u)| +\|u-\tilde u\|_2). 
\end{align*}  
Thus the union bound is also OK for this term, and we have 
\begin{align*}
|\eqref{Sep24eA0.3} - \operatorname{mean} |
\ll 1, \qquad\forall\, \xi \in
\mathbb S^{n-1}, \qquad\forall\, \frac 12 \le \|u\|_2 \le 2.
\end{align*}

\texttt{Estimate of \eqref{Sep24eA0.4} and \eqref{Sep24eA0.5}}. 
We begin by noting that  \eqref{Sep24eA0.4} and \eqref{Sep24eA0.5} can be 
combined into one term. Namely, observe that 
\begin{align*}
 & (\partial_{bb}G)(a,b,c) \cdot 2b +(\partial_b G)(a,b,c) \notag \\
 =& \frac {H_1} {(a+\beta_1 b+\beta_2 c)^3},
 \end{align*}
 where
 \begin{align*}
 H_1&=\beta_1^2 b^3
 +a^2(6b-2c)+3\beta_1\beta_2 b^2c+3b(\beta_1^2+2\beta_1\beta_2
 +2\beta_2^2)c^2-\beta_2(\beta_1+2\beta_2)c^3 \notag \\
 &\quad 
 +a\Bigl(3\beta_1 b^2+6(\beta_1+2\beta_2)bc-(\beta_1+4\beta_2)c^2\Bigr).
 \end{align*}
We can then write
\begin{align*}
 & \eqref{Sep24eA0.4} + \eqref{Sep24eA0.5} \notag\\
 =&\; \frac 1m \sum_{k=1}^m (a_k\cdot \xi)^2 h_3(u, a_k\cdot u, X_k),
 \end{align*}
 where $h_3$ is a bounded smooth function with bounded derivatives in all of its arguments.
 Now let $\phi \in C_c^{\infty}$ be such that $0\le \phi (x) \le 1$ for all $x$, $\phi(x)=1$
 for $|x|\le 1$ and $\phi(x)=0$ for $|x|\ge 2$. We then split the sum as
 \begin{align*}
 & \frac 1m \sum_{k=1}^m (a_k\cdot \xi)^2 h_3(u, a_k\cdot u, X_k),
  \notag \\
 =& \frac 1m \sum_{k=1}^m (a_k\cdot \xi)^2 \phi(\frac{a_k\cdot \xi} K)
  h_3(u, a_k\cdot u, X_k) \notag \\
  & \qquad+ \frac 1m \sum_{k=1}^m (a_k\cdot \xi)^2  (1-\phi(\frac {a_k\cdot \xi} K) )
  \cdot h_3(u, a_k\cdot u, X_k),
 \end{align*}
 where $K$ will be taken sufficiently large. Clearly the first term will be OK
 for union bounds. On the other hand, the second term can be dominated by
 \begin{align*}
 \operatorname{const} \cdot \frac 1m \sum_{k=1}^m
 (a_k\cdot \xi)^2 ( 1- \phi(\frac {a_k\cdot \xi} K) ),
 \end{align*}
 which can be made small by taking $K$ large. 
 Thus we have 
 \begin{align*}
 | \eqref{Sep24eA0.4}+\eqref{Sep24eA0.5}
 -\operatorname{mean} | \ll 1, 
 \quad\forall\, \xi \in \mathbb S^{n-1}, \quad\forall\, \frac 12 \le \|u\|_2\le 2.
 \end{align*}
 
 Collecting the estimates, we have for $m\gtrsim n$ and with high probability,
 \begin{align*}
 \Bigl| \sum_{i,j=1}^n \xi_i \xi_j \partial_{u_i u_j} f(u ) - \operatorname{mean} \Bigr|\ll 1, 
\quad\forall\, \xi \in \mathbb S^{n-1}, \quad\forall\, \frac 12 \le \|u\|_2\le 2.
 \end{align*}
 The desired result then follows from Lemma \ref{lemSeI:2}.
\end{proof}


\begin{thebibliography}{10}

\bibitem{bhojanapalli2016global}
S. Bhojanapalli, N. Behnam, and N. Srebro,
\newblock ``Global optimality of local search for low rank matrix recovery,''
\newblock {\em  Advances in Neural Information Processing Systems},  pp. 3873--3881, 2016.




\bibitem{Phaseliftn}
E. J.  Cand\`es and X. Li,
\newblock ``Solving quadratic equations via PhaseLift when there are about as many equations as unknowns,''
\newblock {\em Found. Comut. Math.}, vol. 14, no. 5, pp. 1017--1026, 2014.


\bibitem{WF}
E. J.  Cand\`es, X. Li, and M. Soltanolkotabi,
\newblock ``Phase retrieval via Wirtinger flow: Theory and algorithms,''
\newblock {\em IEEE Trans. Inf. Theory}, vol. 61, no. 4, pp. 1985--2007,
  2015.

\bibitem{phaselift}
E. J.  Cand\`es, T. Strohmer, and V. Voroninski,
\newblock ``Phaselift: Exact and stable signal recovery from magnitude
  measurements via convex programming,''
\newblock {\em Commun. Pure Appl. Math.},  vol. 66, no. 8, pp. 1241--1274, 2013.


\bibitem{2020a}
J. Cai, M. Huang, D. Li and Y. Wang,
\newblock `` Solving phase retrieval with random initial guess is nearly as good as by spectral
 initialization,"
\newblock {\em Appl. Comput. Harmon.  Anal.}, 2021.

\bibitem{2021b}
J. Cai, M. Huang, D. Li and Y. Wang,
\newblock `` Nearly optimal bounds for the global geometric landscape of phase retrieval,"
\newblock  arxiv preprint, in preparation, 2021.

\bibitem{TWF}
Y. Chen and E. J. Cand\`es,
\newblock `` Solving random quadratic systems of equations is nearly as easy as
  solving linear systems,''
\newblock {\em Commun. Pure Appl. Math.}, vol. 70, no. 5, pp. 822--883, 2017.


\bibitem{fienup1987phase}
J. C. Dainty and J.R. Fienup,
 \newblock ``Phase retrieval and image reconstruction for astronomy,''
 \newblock {\em Image Recovery: Theory and Application}, vol. 231, pp. 275, 1987.


\bibitem{du2017gradient}
S. S. Du, C. Jin, J. D. Lee, and M. I. Jordan,
\newblock ``Gradient descent can take exponential time to escape saddle points,''
\newblock {\em  Advances in Neural Information Processing Systems},  pp. 1067--1077, 2017.


\bibitem{ER3}
J. R. Fienup,  ``Phase retrieval algorithms: a comparison,''
\newblock {\em Appl. Opt.},  vol. 21, no. 15, pp. 2758--2769, 1982.




\bibitem{PAF}
B. Gao, Y. Wang, and Z. Xu,
\newblock Solving a perturbed amplitude-based model for phase retrieval, 2019 [Online].
\newblock Available: http://arxiv.org/abs/1904.10307

\bibitem{Gaoxu}
B. Gao and Z. Xu,
\newblock ``Phaseless recovery using the Gauss--Newton method,''
\newblock {\em IEEE Trans. Signal Process.},  vol. 65, no. 22, pp. 5885--5896,
  2017.

\bibitem{ge2016matrix}
R. Ge, F. Huang, C. Jin, and Y. Yuan,
\newblock ``Escaping from saddle points—online stochastic gradient for tensor decomposition,''
\newblock {\em Conference on Learning Theory},  pp. 797--842,
  2015.

\bibitem{ge2015escaping}
R. Ge, J. Lee, C. Jin, and T. Ma,
\newblock ``Matrix completion has no spurious local minimum,''
\newblock {\em Advances in Neural Information Processing Systems},  pp. 2973--2981,
  2016.




\bibitem{Gerchberg1972}
R. W. Gerchberg,
\newblock “A practical algorithm for the determination of phase from image and diffraction plane pictures,”
\newblock {\em Optik}, vol. 35, pp. 237--246, 1972.
%






\bibitem{gerchberg1972practical}
R. W. Gerchberg and W. O. Saxton,
\newblock ``A practical algorithm for the determination of the phase from image and diffraction plane pictures,''
\newblock {\em Optik}, vol. 35, pp. 237-246, 1972.




\bibitem{harrison1993phase}
R. W. Harrison,
\newblock ``Phase problem in crystallography,''
\newblock {\em JOSA A},  vol. 10, no. 5, pp. 1046--1055, 1993.

\bibitem{huang2021linear}
M. Huang and Y. Wang,
\newblock Linear convergence of randomized Kaczmarz method for solving complex-valued phaseless equations, 2021 [Online].
\newblock Available: http://arxiv.org/abs/2109.11811

\bibitem{jin2017escape}
C. Jin, R. Ge, P. Netrapalli, S. M. Kakade, and M. I. Jordan,
\newblock ``How to escape saddle points efficiently,''
\newblock {\em Proceedings of the 34th International Conference on Machine Learning-Volume 70}, pp. 1724--1732, 2017.


\bibitem{jin2017accelerated}
C. Jin, P. Netrapalli, and M. I. Jordan,
\newblock Accelerated gradient descent escapes saddle points faster than gradient descent, 2017 [Online].
\newblock Available: http://arxiv.org/abs/1711.10456

\bibitem{cai2019}
Z. Li, J. F. Cai, and K. Wei,
\newblock ``Towards the optimal construction of a loss function without spurious local minima for solving quadratic equations,''
\newblock {\em IEEE Trans. Inf. Theory}, vol. 66, no. 5, pp. 3242--3260, 2020.


\bibitem{miao2008extending}
J. Miao, T. Ishikawa, Q. Shen, and T. Earnest,
\newblock ``Extending x-ray crystallography to allow the imaging of
  noncrystalline materials, cells, and single protein complexes,''
\newblock {\em Annu. Rev. Phys. Chem.}, vol. 59, pp. 387--410, 2008.

\bibitem{millane1990phase}
R. P. Millane,
\newblock ``Phase retrieval in crystallography and optics,''
\newblock {\em J. Optical Soc. America A}, vol. 7, no. 3, pp.  394-411, 1990.


\bibitem{AltMin}
P. Netrapalli, P. Jain, and S. Sanghavi,
\newblock ``Phase retrieval using alternating minimization,''
\newblock {\em IEEE Trans. Signal Process.},  vol. 63, no. 18, pp. 4814--4826,
  2015.


\bibitem{park2016non}
D. Park, A. Kyrillidis, and C. Caramanis,
\newblock Non-square matrix sensing without spurious local minima via the Burer-Monteiro approach, 2016 [Online].
\newblock Available: http://arxiv.org/abs/1609.03240
%
%
%
%

\bibitem{Sahinoglou}
H. Sahinoglou and S. D. Cabrera,
\newblock ``On phase retrieval of finite-length sequences using the initial time sample,''
\newblock {\em IEEE Trans. Circuits and Syst.},  vol. 38, no. 8, pp. 954--958, 1991.
%

\bibitem{shechtman2015phase}
Y. Shechtman, Y. C. Eldar, O. Cohen, H. N. Chapman, J. Miao, and M. Segev,
\newblock ``Phase retrieval with application to optical imaging: a contemporary
  overview,''
\newblock {\em IEEE Signal Process. Mag.}, vol. 32, no. 3, pp. 87--109, 2015.

\bibitem{Sun18}
J. Sun, Q. Qu, and J, Wright,
\newblock ``A geometric analysis of phase retrieval,''
\newblock {\em Found. Comput. Math.}, vol. 18, no. 5, pp.  1131--1198, 2018.

\bibitem{sun2016complete}
J. Sun, Q. Qu, and J, Wright,
\newblock ``Complete dictionary recovery over the sphere I: Overview and the geometric picture,''
\newblock {\em IEEE Trans. Inf. Theory}, vol. 63, no. 2, pp.  853--884, 2016.


\bibitem{tan2019phase}
Y. S.  Tan and R. Vershynin,
\newblock ``Phase retrieval via randomized kaczmarz: Theoretical guarantees,''
\newblock {\em Information and Inference: A Journal of the IMA},  vol. 8, no. 1, pp.  97--123, 2019.


\bibitem{Vershynin2018}
R. Vershynin,
\newblock {\em High-dimensional probability: An introduction with applications
  in data science.}
\newblock U.K.:Cambridge Univ. Press, 2018.




\bibitem{Waldspurger2015}
I. Waldspurger, A. d'Aspremont, and S. Mallat,
\newblock ``Phase recovery, maxcut and complex semidefinite programming,''
\newblock {\em Math. Prog.},  vol. 149, no. 1-2, pp. 47--81, 2015.


\bibitem{walther1963question}
A. Walther,
\newblock ``The question of phase retrieval in optics,''
\newblock {\em J. Mod. Opt.}, vol. 10, no. 1, pp. 41--49, 1963.

\bibitem{TAF}
G. Wang, G. B. Giannakis, and Y.~C. Eldar,
\newblock ``Solving systems of random quadratic equations via truncated amplitude
  flow,''
\newblock {\em IEEE Trans. Inf. Theory},  vol. 64, no. 2, pp. 773--794, 2018.


\bibitem{Wei2015} 
K. Wei,
\newblock ``Solving systems of phaseless equations via kaczmarz methods: a proof of concept study,''
\newblock {\em  Inverse Probl.},  vol. 31, no. 12, 125008, 2015.


\bibitem{RWF}
H. Zhang, Y. Zhou, Y. Liang, and Y. Chi,
\newblock ``A nonconvex approach for phase retrieval: Reshaped wirtinger flow and incremental algorithms,''
\newblock {\em The Journal of Machine Learning Research}, vol. 18, no. 1, pp. 5164--5198, 2017.




\end{thebibliography}
\end{document}